\documentclass[10pt,reqno]{amsart}
 
\usepackage[T1]{fontenc}
\usepackage{slashed}
\usepackage{amssymb}
\usepackage{amsmath,amsthm}
\usepackage{amsfonts}
\usepackage{leftidx}
\usepackage{mathrsfs}
\usepackage{enumerate}	
\usepackage{abstract}
\usepackage{stmaryrd}
\usepackage{ulem}
\usepackage{graphicx}
 \usepackage{hyperref}

\usepackage[top=0.8in, bottom=0.8in, left=0.6in, right=0.6in]{geometry}

\def\R{\mathbb{R}}

\def\d{|\nabla|}

\def\Z{\mathbb{Z}}

\def\p{\partial}

\def\vo{\vspace{1\baselineskip}}

\def\be{\begin{equation}}
\def\ee{\end{equation}}

\newtheorem{theorem}{Theorem}[section]

\newtheorem{lemma}{Lemma}[section]
\newtheorem{proposition}{Proposition}[section]
\theoremstyle{definition}

\theoremstyle{remark}
\newtheorem{remark}{Remark}[section]

\numberwithin{equation}{section}
\begin{document}
 \title[$3D$ simplified RVM]{ Remarks on the Large data global solutions of $3D$ RVP  system and $3D$ RVM system}
\author{Xuecheng Wang}
\address{YMSC, Tsinghua  University, Beijing, China,  100084
}
\email{xuecheng@tsinghua.edu.cn,\quad xuechengthu@outlook.com 
}

\thanks{}
 
\maketitle
\begin{abstract}
 We show that  the simplified  $3D$ relativistic Vlasov-Maxwell (sRVM) system, in which there is no magnetic field, poses a global solution for a class of  arbitrarily large cylindrically symmetric initial data. In particular, a vanishing order condition imposed on the initial data  for the relativistic Vlasov-Poisson system (RVP) in \cite{wang2}  is not imposed for the  sRVM system. 
\end{abstract}
 
\section{Introduction} 

\subsection{Motivation}
The  $3D$ relativistic  Vlasov-Maxwell system is one of the   fundamental models in the collisionless plasma physics. It reads as follows, 
\be\label{mainequation}
(\textup{RVM})\qquad \left\{\begin{array}{c}
\p_t f + \hat{v} \cdot \nabla_x f + (E+ \hat{v}\times B)\cdot \nabla_v f =0,\\
\nabla \cdot E = 4 \pi \displaystyle{\int_{\R^3}  f(t, x, v) d v}, \qquad \nabla \cdot B =0, \\
\p_t E = \nabla \times B - 4\pi \displaystyle{\int_{\R^3} f(t, x, v) \hat{v} d v }, \quad \p_t B  =- \nabla\times E,\\
f(0,x,v)=f_0(x,v), \quad E(0,x)=E_0(x), \quad B(0,x)=B_0(x),
\end{array}\right.
\ee
where $f:\R_t\times \R_x^3\times \R_v^3\longrightarrow \R_{+}$ denotes the distribution function of particles, $E, B :\R_t\times \R_x^3 \longrightarrow \R^3 $ denote  the electromagnetic field,  $\hat{v}:=v/\sqrt{1+|v|^2}$.

In particular, if there is no magnetic field, then   the electric field is curl free. Hence, we know that $E=\nabla\phi$, where $\phi:\R_t\times \R_x^3\longrightarrow \R$. As a result, the RVM system is reduced to  the following simplified RVM (sRVM) system 
\be\label{sRVM}
(\textup{sRVM})\quad \left\{\begin{array}{c}
\p_t f + \hat{v} \cdot \nabla_x f + E \cdot \nabla_v f =0,\quad E=\nabla\phi \\
\Delta \phi = 4 \pi \displaystyle{\int_{\R^3}  f(t, x, v) d v},\qquad
\p_t\nabla_x\phi =  - 4\pi \displaystyle{\int_{\R^3} f(t, x, v) \hat{v} d v }, \quad 
f(0,x,v)=f_0(x,v). 
\end{array}\right. 
\ee
The following conservation law holds for the above system, 
\be\label{conservationlaw}
\mathcal{H}(t):= \int_{\R^3} \int_{\R^3} \sqrt{1+|v|^2} f(t,x,v) dx d v + \int_{\R^3} |\nabla_x \phi(t,x)|^2 dx = \mathcal{H}(0).
\ee

We remark that, due to the divergence free condition of the magnetic field,  for any smooth localized radial initial data,  the solution of the  RVM system (\ref{mainequation}) must have zero magnetic field.

As comparison, the following $3D$ relativistic Vlasov-Poisson system (RVP) for the plasma physics case  is also widely studied in the literature. 
 \be\label{vlasovpo}
(\textup{RVP})\quad \left\{\begin{array}{c} 
\p_t f + \hat{v} \cdot \nabla_x f + E \cdot \nabla_v f =0,\quad E=\nabla\phi \\
\Delta \phi = \rho(t), \quad \rho(t):= \displaystyle{\int_{\R^3} f(t,x,v) d v},\quad f(0,x,v)=f_0(x,v). \\ 
\end{array}\right. 
\ee

Note that the RVP is a determined system  while sRVM system is over-determined in general.  However,  for the radial case,  these two systems are equivalent. 

 The properties of the solution of the  RVP system are shared with the solution of the sRVM system, however not the other way around. For example, the solution of sRVM system (\ref{sRVM}) enjoys the following property over time, 
\[
\int_{\R^3} \hat{v}\times \nabla_x f(t,x,v) d v =0,   
\]
which is  not necessarily true for the solution of the RVP system in general. But, indeed, it's true for the radial case. 

In \cite{wang2}, we show the global existence   of $3D$ RVP system (\ref{vlasovpo})   for a class of arbitrarily large   initial data with cylindrical symmetry and an extra vanishing order condition on the planar momentum. More precisely, the following condition is imposed on the initial data, 
\[
\textup{(a vanishing order condition for RVP)}\qquad \int_{\R^3}\int_{\R^3} \frac{1}{|\slashed x \times \slashed v|^{13}} f_0(x,v) d x d v  < +\infty.
\]

In \cite{wang2}, we exploit the benefit of the above condition via  the conservation law of planar momentum of the $3D$ RVP system for the cylindrical symmetry. The the conservation law of planar momentum is available is mainly because of the cylindrical symmetry assumption and the gradient structure of the electric field. However, the gradient structure of the electric field is no longer available for the RVM system (\ref{mainequation}). We don't expect such conservation law for the RVM system. 

One of the main goals of this paper is to show that, thanks to the extra equation in sRVM system (\ref{sRVM}),  the difficulty of removing the vanishing order condition for the RVP system is absent in  the  sRVM system as well as in  the RVM system, see \cite{wang3}.  In this paper, we will not use the conservation law of the planar momentum of any kind to show large data global existence of the sRVM system (\ref{sRVM}). Instead, we exploit the hyperbolic nature of the sRVM system.

We also reveal some properties  shared commonly by the RVP system and the sRVM system, e.g., the estimate (\ref{projest}) in Lemma \ref{horprojlemma}. We hope that it can shine some light on the further study of the large data problem of the RVP system (\ref{vlasovpo}).  

Moreover, as a toy model for the RVM system (\ref{mainequation}), the framework of proving global existence used in this paper is same as the framework used in a  more complicated study of the RVM system   \cite{wang3}.  For  pedagogical purpose, we  also hope that this paper also serves as an introduction for more complicated and more technical full study of the RVM system in \cite{wang3}.

\subsection{Main result of this paper}

The main result of this paper is stated as follows. 
 
\begin{theorem}\label{maintheorem}
Assume that the   initial data  $  f_0(x,v)\in H^s(\R_x^3 \times \R_v^3)$, $s\in \mathbb{Z}_{+},s\geq 6$ is cylindrical symmetric  in the sense that the following equality holds for any $x,v\in \R^3$, 
\be\label{april25eqn1}
f_0(x_1,x_2,x_3,v_1,v_2,v_3)= f_0(|\slashed{x}|,0,x_3,|\slashed{v}|,0,v_3),\quad  \slashed{x}:=(x_1,x_2), \quad  \slashed{v}:=(v_1,v_2). 
\ee
Moreover,   we assume that    the initial data decays polynomially as $(x,v)\longrightarrow \infty$ in the following sense,   
\be\label{assumptiononinitialdata}
\sum_{\alpha \in \mathbb{Z}_{+}^6,|\alpha|\leq s} \|(1+| x|+|v|)^{N_0}\nabla_{x,v}^\alpha f_0(x,v)\|_{L^2_{x,v}}< +\infty, \quad  N_0:= 10^{8},
\ee
then the simplified  relativistic Vlasov-Maxwell system \textup{(\ref{sRVM})} admits a global solution $ (f(t,\cdot, \cdot) )\in  H^s(\R_x^3 \times \R_v^3)$. 

\end{theorem} 

A few remarks are in order. 

\begin{remark}
Although, as stated in (\ref{april25eqn1}),   the initial data is cylindrical symmetric with respect to the  $x_3$-axis, by changing coordinate  system, our main theorem also holds for  initial data  with   cylindrical symmetry with  respect to any  line  that goes through the origin. 
\end{remark}

\begin{remark}
The plausible  goal of optimizing the size of $N_0, s $ is not pursued here. 
\end{remark}

\subsection{Notation and plan of this paper}
 For any two numbers $A$ and $B$, we use  $A\lesssim B$ and $B\gtrsim A$  to denote  $A\leq C B$, where $C$ is an absolute constant. We use the convention that all constants which only depend on the \textbf{initial data}, e.g., the conserved quantities in  (\ref{conservationlaw}),  will be treated as absolute constants.  

   We  fix an even smooth function $\tilde{\psi}:\R \rightarrow [0,1]$, which is supported in $[-3/2,3/2]$ and equals to ``$1$'' in $[-5/4, 5/4]$. For any $k,j\in \mathbb{Z}, j> 0$, we define the cutoff functions $\psi_k, \psi_{\leq k}, \psi_{\geq k}:\cup_{n\in \Z_+}\R^n\longrightarrow \R$ as follows, 
\[
\psi_{k}(x) := \tilde{\psi}(|x|/2^k) -\tilde{\psi}(|x|/2^{k-1}), \quad \varphi_0(x):=  \tilde{\psi}(|x|), \quad\forall j\in (0, \infty)\cap \Z,\quad  \varphi_j(x):=\psi_j(x) .
\]

For any $x\in \R^3/\{0\}, u\in \R^3$, we use $\tilde{x}:=x/|x|$ to denote the direction of the vector $x$ and use $\slashed u$ to denote the projection of $u$ onto the $x_1x_2$-plane. For an integrable function $f(x)$, we use both $\widehat{f}(\xi)$ and $\mathcal{F}(f)(\xi)$ to denote the Fourier transform of $f$, which is defined as follows,
\[
\mathcal{F}(f)(\xi)= \int e^{-ix \cdot \xi} f(x) d x.
\]

We use $\mathcal{F}^{-1}(g)$ to denote the inverse Fourier transform of $g(\xi)$.

\vo

The rest of this paper is organized as follows. 

\begin{enumerate}
\item[$\bullet$] In section \ref{preli}, we introduce the set-up of problem  and  record some tools developed in \cite{wang2}. 
\item[$\bullet$] In section \ref{proofmain}, under the assumption of the validity of Lemma \ref{horprojlemma} and the Lemma \ref{bootstraplemma2}, we control both the projection of the velocity characteristics onto the horizontal plane and the full velocity characteristics that start from the $t$-majority set $R_t(0)$, see (\ref{may10majority}). With strong control of $\beta_t$, see Proposition \ref{controlchar},  we control the high momentum of the distribution function over time and prove  theorem \ref{maintheorem}.

\item[$\bullet$] In section \ref{mainlemma1}, we prove  Lemma \ref{horprojlemma}. 

\item[$\bullet$] In section \ref{mainlemma2}, by exploiting the hyperbolic nature of the  sRVM system (\ref{sRVM}),  we prove  Lemma \ref{bootstraplemma2}. 
\end{enumerate}

\vo 

\noindent \textbf{Acknowledgment}\qquad The author acknowledges  support from  NSFC-11801299, 12141102, and MOST-2020YFA0713003. 

\section{Preliminary}\label{preli}

\subsection{The set-up}
We use the classic moment method, which is propagating a    high order moment of the distribution function over time, to control directly the electric field over time, see Lemma \ref{roughesti}. We define 
 \be\label{may2eqn1}
M_{ n}(t ):= \int_{\R^3} \int_{\R^3} (1+ |v|)^n f(t,x,v) d v d x,\quad    \tilde{M}_n(t) := (1+t)^{n^2}+\sup_{s\in[0,t]} M_n(s) , \quad n:= N_0/10=10^7. 
 \ee
Therefore $\tilde{M}_n(t) $ is an increasing function with respect to time.

Let $T\in \R_+$ be the maximal time of existence.   For any fixed $t\in[0,T]$, we let $M_t \in \mathbb{Z}_+$ to be the least integer such that $2^{M_t}\geq (\tilde{M}_n(t) )^{1/  (n-1 )}$.  Moreover, it would be sufficient to only consider the case that $M_t\gg 1 $ otherwise  the momentum $\tilde{M}_n(t)$ is naturally bounded by an absolute constant by definition.  

From (\ref{sRVM}), the backward characteristics associated with the sRVM system  read  as follow, 
\be\label{characteristicseqn}
\left\{\begin{array}{l}
\displaystyle{\frac{d}{ds} X(x,v,s,t) = \hat{V}(x,v,s,t)}, \quad 
\displaystyle{\frac{d}{ds} V(x,v, s,t) = E(s, X(x,v,s,t))}\\ 
\\ 
X(x,v,t,t)=x, \quad V(x,v,t,t)=v.\\ 
\end{array}\right. 
\ee
Due to the transport nature of the Vlasov equation, see (\ref{sRVM}),  for any $s, t\in [0, T^{  })$, we have
\be\label{conservation}
f(t,x,v) = f(s, X(x,v,s,t), V(x,v,s,t)).
\ee

 For any $t\in [0, T^{  })$,   we define  $t$-majority set  of particles, which initially localize around zero, at time $s\in [0, T^{  })$ as follows,  
\be\label{may10majority}
R_t(s):= \{(X(x,v,s,0),V(x,v,s,0)): |V (x,v,0,0)|+| X(x,v,0,0) |\leq  2^{ M_t/2}  \}.
\ee
Moreover, we define 
\begin{multline}\label{may9en21}
\alpha_t(s,x,v):=  \sup_{\tau \in[0,s] } \inf\{ k: k\in \R_+, |  \slashed V(x,v,\tau,0)| \leq 2^{k M_t  }   \}  ,\quad \alpha_t(s):=\sup_{(x,v)\in R_t(0)} \alpha_t(s,x,v), \\ 
\beta_t(s, x,v):= \sup_{\tau\in[0,s] } \inf\{ k: k\in \R_+, | V(x,v,\tau,0)|\leq 2^{k M_t}  \}, \quad \beta_t(s):=\sup_{(x,v)\in R_t(0)} \beta_t(s,x,v), \\ 
 \epsilon:=100/n=10^{-5}, \quad \alpha_t:=\alpha_t(t), \quad  \tilde{\alpha}_t=\min\{\alpha_t, 1+\epsilon \},\quad  \beta_t:=\beta_t(t), \quad  \tilde{\beta}_t=\min\{\beta_t, 1+\epsilon \}.
\end{multline}
From the above definition, we know that 
\be\label{2021dec18eqn1}
\forall t\in [0, T^{  }), s\in [0, t], \quad R_s(0)\subset R_t(0), \quad  \alpha_s M_s \leq \alpha_t(s) M_t\leq \alpha_t M_t, \quad \beta_s M_s\leq \beta_t(s) M_t\leq \beta_t M_t. 
\ee

Note that
\[
\forall s\in[0, t], x, v \in \R^3, \quad X( X(x,v,0,s), V(x,v,0,s),s,0)= x, \quad V( X(x,v,0,s), V(x,v,0,s),s,0)= v. 
\]
Therefore, from the above definition,  for any fixed  $s\in [0, t],  x \in \R^3, v\in \R^3$, s.t., either $|\slashed v|\geq 2^{(\alpha_s+\epsilon) M_s }$ or $| v|\geq 2^{(\beta_s+\epsilon ) M_s}$, we have
\be\label{nov24eqn27}
 (X(x,v,0,s), V(x,v,0,s))\notin R_s(0). 
\ee

 Thanks to the rapid polynomial decay rate of the initial data, see (\ref{assumptiononinitialdata}),  to control the high moments       $ \tilde{M}_n(t) $, it would be sufficient    to   control the majority  set over time. More precisely, from the assumption on the initial data in (\ref{assumptiononinitialdata}), the following estimate holds for any $s\in [0, t],  x,v\in \R^3$, s.t., either $|\slashed v|\geq 2^{ (\alpha_s+\epsilon) M_s}$ or $| v|\geq 2^{ (\beta_s+\epsilon) M_s}$, 
 \be\label{outsidemajority}
 \big|f(s,x,v)\big| = \big|f_0( X(x,v,0,s),V(x,v,0,s) )\big|\lesssim 2^{-4 n M_s}. 
 \ee

To exploit the benefit of the hyperbolic natural of the sRVM system, we will use the following Kirchhoff's formula. 

 \begin{lemma}[Kirchhoff's formula]\label{Kirchhoff}
 For any $t\in \R,x\in \R^3$, the following equality holds, 
\be\label{dec20eqn1}
\d^{-1}   \sin( t\d ) h( x)  = \frac{1}{4\pi}   t  \int_{\mathbb{S}^2} h(x+t \theta) d \theta,
\ee
\be\label{march4eqn100}
\d^{-1}   \cos(t\d) h(x)= \frac{1}{4\pi}     \int_{\mathbb{S}^2}  \d^{-1}  h(x+t \theta) d \theta     + \frac{1}{4\pi}     \int_{\mathbb{S}^2} t \theta\cdot \frac{\nabla}{\d} h(x+t \theta) d \theta. 
\ee
\end{lemma}
\begin{proof}
Note that 
\be\label{march4eqn41}
 \int_{\R^3} e^{-i   x\cdot \xi } \int_{\mathbb{S}^2} h(x+t\theta) d \theta d x = \int_{\mathbb{S}^2 } e^{i  t \xi \cdot \theta} \hat{h}(\xi) d\theta=  2\pi \hat{h}(\xi)  \int_{0}^{\pi}e^{it|\xi|\cos(\phi)} \sin(\phi) d \phi =  \frac{4\pi \sin(  t |\xi|)}{t|\xi|}  \hat{h}(\xi)  .
\ee
 Hence finishing the proof of the desired formula (\ref{dec20eqn1}). Our desired equality (\ref{march4eqn100}) holds after taking derivative with respect to ``$t$'' for  the equality (\ref{dec20eqn1}).
\end{proof}

\subsection{Rough estimates for the electric field}
Let  $E_k(\cdot, \cdot)=P_k(E)(\cdot, \cdot)$ denotes the frequency localized electric field.  Moreover, we localize further the sizes of $\slashed v $ and $v$ for the frequency localized electric field as follows,  
 \begin{multline}\label{june9eqn2} 
 E_k(t,x)= \sum_{j_2\in \Z_+, j_1\in [0, j_2+2]\cap \Z_+}  E_{k;j_1,j_2}(t,x),\\ 
 E_{k;j_1,j_2}(t,x)= \int_{\R^3} \int_{\R^3} K_k(y) f(t,x-y, v) \varphi_{j_1}(\slashed v)\varphi_{j_2}(v) d y d v,\quad  K_k(y)= \int_{\R^3} \int_{\R^3} e^{i y\cdot \xi}\frac{-i\xi}{|\xi|^2} \varphi_{k}(\xi)   d \xi. 
 \end{multline}
After doing integration by parts in $\xi$ many times, the following estimate holds for the kernel  $K_k(y)$, 
\be\label{june9eqn21}
 \big|K_k(y)\big|\lesssim 2^{2k}(1+2^{k}|y|)^{-N_0^3}. 
\ee

The following singular weighted space-time estimate controls strongly the distribution of particles near the $x_3$-axis. 
\begin{lemma}\label{spacetimeest}
Let $\epsilon^{\star}:=\epsilon/100.$ For any $t\in [0, T^{})$, the following weighted space-time estimate holds,
 \be\label{jan12eqn36}
 A(t):= \int_0^t\int_{\R^3} \int_{\R^3}  \frac{| \slashed v|^{2+2\epsilon^{\star}}}{|\slashed x|^{1-2\epsilon^{\star}}\langle v \rangle }    f(s,x,v)  d x  d v d s \lesssim  2^{5\epsilon M_t}.
 \ee
 As a by-product, the following $L^1_t-L^\infty_x$-type estimate holds for any $t_1, t_2\in[0,t]$, 
 \be\label{june9eqn10}
\int_{t_1}^{t_2} \| E_{k;j_1,j_2}(t,\cdot )\|_{L^\infty_x} d  t \lesssim 2^{2\epsilon M_t} + 2^{k-2j_1+j_2+6\epsilon M_t}. 
 \ee

\end{lemma}
\begin{proof}
See \cite{wang2}[Proposition 3.1].
\end{proof}

We record the following rough estimates for the localized electric field obtained in \cite{wang2}.
 \begin{lemma}\label{roughesti}
 For any $s\in [0, t]\subset [0, T^{})$,  the following rough estimate holds for the localized electric field, 
\be\label{may31eqn1}
\|E_{k;j_1,j_2}(s,x)\|_{L^\infty_x} \lesssim \min\{2^{-k+2j_1+j_2}, 2^{2k-j_2}, 2^{2k-n j_2} \tilde{M}(t)\}. 
\ee
Moreover, for any $x\in \R^3$ s.t., $|\slashed x|\neq 0$, we have the following point-wise estimate, 
\be\label{may31eqn2}
|E_{k;j_1,j_2}( s,x)|\lesssim 1+ \min\{ \frac{ 2^{j_1+\epsilon M_t}}{|\slashed x |^{1/2}}, \frac{2^{k-j_2+\epsilon M_t}}{|\slashed x |} \}.
\ee
\end{lemma}
\begin{proof}
See \cite{wang2}[Lemma 3.1]. 

\end{proof}
\begin{lemma}\label{roughgeneral}
 For any $s, s_1,s_2\in [0, t]\subset [0, T^{})$,  the following $L^1_t L^\infty_x$-type and $L^\infty_x$ type rough estimates for the electric field hold, 
 \be\label{june14eqn21}
  \int_{s_1}^{s_2} \| E_{ }(s,\cdot )\|_{L^\infty_x} d  s \lesssim2^{(1+5\epsilon)M_t}, 
 \ee
 \be\label{june2eqn71}
 \|E_{ }(s,x)\|_{L^\infty_x} \lesssim  \min\{ 2^{4  \tilde{\alpha}_t M_t/3 + M_t/3+5\epsilon M_t},  |\slashed x |^{-1/2} 2^{ \tilde{\alpha}_t M_t+\epsilon M_t }\}.  
\ee
\end{lemma}
\begin{proof}
Recall the decomposition of the electric field in (\ref{june9eqn2}). 
From the obtained rough estimate (\ref{may31eqn1}), we can first rule out the case $j_2\geq (1+\epsilon)M_t$ as follows, 
\[
\sum_{j_2\geq (1+\epsilon)M_t}\sum_{j_1\in[0,j_2+2]\cap \Z} \sum_{k\in \Z}   \| E_{k;j_1,j_2 }(s,\cdot )\|_{L^\infty_x} \lesssim \sum_{k\in \Z}  \sum_{j_2\geq (1+\epsilon)M_t}2^{\epsilon j_2}  \min\{2^{-k+3j_2}, 2^{2k-nj_2+(n-1)M_t} \}
\]
\be\label{2022jan18eqn30}
\lesssim  \sum_{k\in \Z}   \sum_{j_2\geq (1+\epsilon)M_t} 2^{-\epsilon k} 2^{ (2+\epsilon)j_2 }\big( 2^{-nj_2+(n-1)M_t}\big)^{(1-\epsilon)/3}\lesssim 1. 
\ee
For any fixed $j_2\in [0, (1+\epsilon)M_t]\cap \Z$, again from the  rough estimate (\ref{may31eqn1}), we can rule out the case $k\geq 3M_t + 5\epsilon M_t$ as follows, 
\be\label{2022jan18eqn31}
\sum_{k\in \Z, k\geq  3M_t + 5\epsilon M_t}  \| E_{k;j_1,j_2 }(s,\cdot )\|_{L^\infty_x}\lesssim \sum_{k\in \Z, k\geq  3M_t + 5\epsilon M_t} 2^{-k+3(1+\epsilon)M_t} \lesssim 1. 
\ee

It remains to consider the case $k\in [0,3M_t + 5\epsilon M_t]\cap \Z$,  $j_2\in [0, (1+\epsilon)M_t]\cap \Z $ and $j_1\in [0, j_2+2]\cap \Z$. Since there are at most $M_t^3$ cases, which causes only logarithmic loss. It would be sufficient to let $k,j_1,j_2$ be fixed. From the estimate (\ref{june9eqn10}) and the estimate (\ref{may31eqn1}), we have
\be\label{2022jan18eqn71}
 \int_{s_1}^{s_2}\| E_{k;j_1,j_2 }(s,\cdot )\|_{L^\infty_x}  ds \lesssim 2^{2\epsilon M_t} + \min\{2^{-k+2j_1+j_2+\epsilon M_t}, 2^{k-2j_1+j_2 + 6\epsilon M_t}\}\lesssim 2^{j_2 + 3.5\epsilon M_t}\lesssim 2^{(1+4.5\epsilon)M_t}. 
\ee
Hence finishing the proof of our desired estimate (\ref{june14eqn21}).

It remains to prove (\ref{june2eqn71}). Note that,  If $j_1 \in [ {\alpha}_t M_t + \epsilon M_t, j_2+2]\cap \Z$, from  the estimates (\ref{2021dec18eqn1}) and  (\ref{outsidemajority}), we have 
\be
f(t,x,v)=f_0( X(x,v,0,t),V(x,v,0,t) )\lesssim 2^{-10M_t}, \quad   
\Longrightarrow \| E_{k;j_1,j_2 }(s,\cdot )\|_{L^\infty_x}  \lesssim 2^{-k + 3j_2-10 M_t}  \lesssim 1. 
\ee
If $j_1 \leq  \min\{{\alpha}_t M_t + \epsilon M_t, j_2+2\} $, then from the rough estimate (\ref{may31eqn1}) in Lemma \ref{roughesti}, we have
\[
\| E_{k;j_1,j_2 }(s,\cdot )\|_{L^\infty_x} \lesssim \min\{2^{-k+2j_1+j_2}, 2^{2k-j_2}\}\lesssim 2^{4j_1/3+j_2/3}\lesssim   2^{4  \tilde{\alpha}_t M_t/3 + M_t/3+3\epsilon M_t} .
\]
Hence finishing the proof of the desired estimate (\ref{june2eqn71}) holds from the above obtained estimates and the estimate (\ref{may31eqn2}) in Lemma \ref{roughesti}. 
\end{proof}

\section{Control of the $t$-majority set and the proof of main theorem}\label{proofmain}

 At large scale, the framework of the proof is almost same as the one we used for the RVM system in \cite{wang3}.  Intuitively speaking, the main ideas of proof of theorem \ref{maintheorem} can be summarized as follows, 
 \begin{multline}\label{logic}
  \textup{smoothing effect  }\stackrel{\textup{bootstrap} }{\Longrightarrow} \alpha_t\leq 0.71, \quad 
  \textup{smoothing effect +hyperbolic nature} +  \alpha_t\leq  0.71 \stackrel{\textup{bootstrap} }{\Longrightarrow} \beta_t\leq 1-2\epsilon \\ 
  \Longrightarrow  \tilde{M}_n(t) \lesssim  (1+t)^{n^2} \Longrightarrow E(t)\in L^\infty([0, T)\times \R^3 \overset{\textup{continuation criteria} }{\Longrightarrow  } \textup{global existence}.
 \end{multline}

  First of all, by assuming the validity of   $\beta_t\leq 1-2\epsilon$, which is stated in Proposition \ref{controlchar}, we prove our main theorem.   From the conservation law (\ref{conservationlaw}), we have 
\be\label{march20eqn11}
 \big| \int_{\R^3}\int_{|v|\leq 2^{(1- \epsilon )M_t }}(1+|v|)^{n} f(t,x,v)   d x dv\big|\lesssim  2^{(n-1)(1- \epsilon)M_t}\leq (\tilde{M}_{n}(t))^{ 1- \epsilon }. 
\ee

Moreover,  from the estimate $\beta_t \leq (1-2\epsilon)$, we know that  $|  X(x,v,0,t) |+|  V(x,v,0,t) |\geq 2^{  M_t /2}$ if $|v|\geq 2^{(1- \epsilon )M_t }$. From the decay assumption of the initial data in (\ref{assumptiononinitialdata}) and the following estimate holds if $|v|\geq 2^{(1- \epsilon)M_t }$, 
\be\label{may11eqn31}
|f(t,x,v)|=| f_0(X(x,v,0,t),  X(x,v,0,t))| \lesssim (1+|X(x,v,0,t)|+|V(x,v,0,t)|)^{-N_0+10}\lesssim 2^{-5 n M_t +5}.
\ee

Recall (\ref{characteristicseqn}), from the rough   $L^\infty_x$ estimate of electric field in (\ref{june2eqn71}) in Lemma \ref{roughgeneral}, the following estimates hold if $ |v|\gtrsim 2^{ 2 M_t}, |x|\gtrsim 2^{2\epsilon M_t}$, 
\be\label{may11eqn32}
\big|v-V(x,v,0,t)|\big| \lesssim    2^{5M_t/3+10\epsilon M_t}, \quad \big|x-X(x,v,0,t)|\big| \lesssim    2^{3 \epsilon M_t/2},\quad  \Longrightarrow |v|\sim |V(x,v,0,t)|, \big|x\big|\sim \big|    X(x,v,0,t) \big|.  
\ee  
 
To sum up, after combining the above estimate  (\ref{may11eqn31}), we know that the following estimate holds if $|v|\geq 2^{(1- \epsilon)M_t }$, 
\be
|f(t,x,v)|=| f_0(X(x,v,0,t),  X(x,v,0,t))|\lesssim (1+|x|)^{-4}(1+|v |)^{-n -4}.  
\ee
From the above estimate, we have
\be\label{may11eqn34}
 \big| \int_{\R^3}\int_{|v|\geq 2^{(1- \epsilon)M_t }}(1+|v|)^{n} f(t,x,v)   d x dv\big|\lesssim 1.
\ee
Therefore, recall (\ref{may2eqn1}), from the estimates (\ref{march20eqn11}) and (\ref{may11eqn34}), we know that the following estimate holds for any  $t\in[0,T)$,
\be 
M_{n}(t)\lesssim \big( \tilde{M}_{n}(t) \big)^{  1- \epsilon }.  
\ee
From the above estimate and 
the fact that $\tilde{M}_n(t)$ is an increasing function with respect to $t $, the following estimate holds for any $s\in [0, t]$,
\[
\tilde{M}_{n}(t)=\sup_{s\in [0, t]}M_{n}(s)+  (1+t)^{n^3}  \lesssim \big( \tilde{M}_{n}(  t) \big)^{  1- \epsilon}+  (1+t)^{n^3}, 
 \quad \Longrightarrow \tilde{M}_{n}(t)\lesssim   (1+t)^{n^3}. 
\]

Therefore, $ \tilde{M}_n(t) $ grows at most at rate $(1+t)^{n^3}$ over time. From the rough   $L^\infty_x$ estimate of electric field in (\ref{june2eqn71}) in Lemma \ref{roughgeneral}, we know that    that  the boundedness of $ \tilde{M}_n(t) $ in $[0, T)$ ensures that $E \in L^\infty([0,T^{})\times \R_{x}^3)$, hence the 
   the lifespan  of the system  (\ref{characteristicseqn}) can be extended at any finite time, i.e., the solution exists globally, see also the continuation criteria by Luk-Strain \cite{luk}[Theorem 5.7]. Hence finishing the proof of Theorem \ref{maintheorem}.

   Now, it would be sufficient to prove the   Proposition \ref{controlchar}. For the rest of this section, we mainly  prove     Proposition \ref{controlchar} by assuming the validity of Lemma \ref{horprojlemma} and Lemma \ref{bootstraplemma2}.

\begin{proposition}\label{controlchar}
For any $t\in [0, T)$, s.t., $M_t\gg 1$, we have
\be\label{sizeofmajority}
\alpha_t \leq \alpha^{\star}:=7/10+\iota, \quad \beta_t\leq (1-2\epsilon), \quad \iota:=10^{-2}, \quad \epsilon= 10^{-5}. 
\ee
\end{proposition}
\begin{proof}

The argument presented is very general, which only depends on the analysis of the motion of the characteristics (\ref{characteristicseqn}) by using bootstrap arguments.  Actually, it's   almost same as the one we used in the study of $3D$ RVM system   \cite{wang3}[section 3]. The key difference between RVP system (\ref{vlasovpo}), RVM system (\ref{mainequation}), and the sRVM system (\ref{sRVM}) is that the behavior of the acceleration force  is very different in different settings. We summarize  the effects of acceleration force over time  as black boxes in     Lemma \ref{horprojlemma} and Lemma \ref{bootstraplemma2}.

 For the sake of readers and also for the pedagogical purpose,  instead of referring readers
  to \cite{wang3}[section 3],   we  still  give a full proof here. 

Let  $t\in [0,T)$   be fixed. Recall (\ref{characteristicseqn}).  For convenience in notation, \textbf{ we suppress  the dependence of   characteristics with respect to $( x,v) \in R_t(0)$}.  We define the first time that characteristics almost reaches the threshold as follows, 
\be\label{2021dec18eqn21}
t^{\ast}:=\inf\{s: s\in [0, t], \beta_t(s)\geq  (1-3\epsilon)  \}. 
\ee
 Due to the continuity of characteristics,  and the rough estimate of the electric  field (\ref{june2eqn71}) in Lemma  \ref{roughgeneral}, we know that $t^{\ast}>0. $ Moreover,  $\forall s \in [0, t^{\ast}],$ we have $|V(s)|\leq 2^{(1-3\epsilon)M_t}. $  We make the following bootstrap assumption, 
\be\label{2021dec18eqn22}
\tau:= \sup\{ \kappa: \kappa \in [0,t], \forall s\in [t^{\ast}, \kappa],   | 2^{ \beta_t(s) M_t}|\in  2^{(1-3 \epsilon)M_t }    [99/100, 101/100] \} .
\ee
We aim to improve the above bootstrap assumption and  prove that $\tau=t$.

  \textbf{Step} $\mathbf{1}$ \qquad Estimate of the projection of the velocity characteristics, i.e., $|\slashed V(s)|$ within the time interval $[0, \tau]$. 

 Let $ \gamma_1:=    \alpha^{\star}   -2  \epsilon$. Define $\tau^1$ and  $\tau^2 $ be the first time and the second time  that  the projection of the velocity characteristics almost reaches a threshold as follows 
\be\label{may10eqn109}
\tau_{\ast}:= \sup\{ s: s\in [0,\tau], \forall\kappa\in [0, s],    {\alpha_t (\kappa) }\leq  \gamma_1\},\quad \tau^{\star}:= \sup\{  s: s\in [0,\tau], \forall\kappa\in [0, s],    {\alpha_t (\kappa) }\leq    \gamma_1   +1/M_t \} 
 .
\ee
Our goal is to show that $\tau^{\star}= \tau$.  If either $\tau_{\ast}$ or $\tau^{\star}$ equals to $\tau$, then there is nothing left to be proved. We focus on the case $\tau_{\ast}< \tau^{\star}< \tau$. Let 
\[
\gamma_2:= \inf\{k: k\in \R_+, |V(x,v,\tau_{\ast}, 0)||\leq 2^{k M_t}\}. 
\]
We make the    bootstrap assumption for the velocity characteristics  $  V(s) $ as follow , 
\be\label{bootstrap1}
  \tau^{\ast}:=\sup\{s: s\in[\tau_{\ast},\tau^{\star} ], \forall \kappa\in [\tau^{\ast}, s],    |\slashed V(s)| \in   2^{  \gamma_1  M_t}  [99/100, 101/100], |V(s)|  \in   2^{   \gamma_2 M_t}  [99/100, 101/100] \}.
\ee
Now, our goal is to show that, independent of $(x,v)\in R_t(0)$, we can improve the above bootstrap assumption and show that $\tau^{\ast}= \tau^{\star}$.  Hence improving the upper bound in the definition of $\tau^{\star}$ and  ruling out the case $\tau^{\star}< \tau.$

Recall (\ref{may10majority}) and (\ref{characteristicseqn}), we have 
\[
\forall s\in [0,t], \quad |X(s)|\leq |X(0)|+\int_{0}^t  |\hat{V}(s)| d s\leq 2^{M_t/2 +1 }.
\]

Based on the possible size of  the magnitude of space characteristic in the horizontal plane, we decompose the interval $[0,M_t/2+1]$ into a finite union of intervals, which overlap with at most with four other intervals. More precisely, 
\begin{multline}\label{overlapintervals}
[0,2^{M_t/2+1}]=\cup_{i=0}^K I_i, \quad I_0=[0,  2^{{\kappa_0 }}], \quad I_k=[2^{a_k},2^{a_k+10}], \quad a_{0}=\kappa_0-6, \kappa_0=-100M_t,  \\ 
 \forall k\in\{1,\cdots, K-2\}, a_{k+1}= a_{k}+6,\quad a_{K}= M_t/20+1-10,\quad   K=\lceil (M_t/2+1-\kappa_0)/6\rceil.
\end{multline}

Moreover, as a result of direct computation, from (\ref{characteristicseqn}),   we have
\begin{multline}\label{oct22eqn1}
\frac{d}{ds} |\slashed X(s)|^2 = \frac{2 \slashed X(s)\cdot \slashed V(s)}{\sqrt{1+|V(s)|^2}},\qquad  \frac{d}{ds} \frac{\slashed X(s)\cdot \slashed V(s)}{\sqrt{1+|V(s)|^2}} = \frac{|\slashed V(s)|^2 }{1+|V(s)|^2}  +\big(\frac{\big(\slashed X(s), 0\big)}{\sqrt{1+|V(s)|^2}}- \frac{\slashed X(s)\cdot \slashed V(s) }{  \big({1+|V(s)|^2}\big)^{3/2}} V(s)\big)\cdot E(s, X(s)),\\ 
\frac{d}{ds} |\slashed V(s)|  = |\slashed V(s)|^{-1}\big(\slashed V(s), 0\big) \cdot E(s, X(s)  ), \quad \frac{d}{ds} |  V(s)|^2  =   2  V(s)  \cdot E(s, X(s)). 
\end{multline}

Within the time interval $[\tau_{\ast}, \tau^{\ast}]$, we   show two facts about the characteristics under the bootstrap assumption (\ref{bootstrap1}).
\begin{enumerate}
\item[ \textbf{Fact (i)}] Given any $t_1, t_2\in  [\tau_{\ast}, \tau^{\ast}]$ and any $i\in\{0,\cdots, K-1\}$, s.t., $\forall s\in [t_1, t_2], |\slashed X(s)|\in I_i\cup I_{i+1}$, then the bounds of speed characteristics $|\slashed V(s)|, |V(s)|$ can be improved for any $s\in [t_1,t_2].$ More, precisely, the increment of  $|\slashed V(s)|$     is at most  $2^{ \gamma_1 M_t- \epsilon M_t}$.  and  the increment of   $|V(s)|$   is at most  $2^{ \gamma_2 M_t- \epsilon M_t}$. 

\item[\textbf{Fact (ii)}] Let $I_{-1}:=\emptyset$. For any  $i\in\{0,\cdots, K\}$ and any $s\in  [\tau_{\ast}, \tau^{\ast}]$, if $|\slashed X(s)|\in I_i/I_{i-1}$ with $d |\slashed X(t)|^2/dt\big|_{t=s}\geq 0 $, then for any later time $l \in[s,\tau^{\ast}]$,   $|\slashed X(l)|\in \cup_{i\leq j\leq K} I_j$.
\end{enumerate}

Based on the possible size of $i$, we divide the proof of   \textbf{Fact (i)}  into two cases as follows. The   \textbf{Fact (ii)} will be obtained as a byproduct.

\quad \textbf{Case} $1$:\quad  Given any $t_1, t_2\in  [\tau_{\ast}, \tau^{\ast}]$, if $\forall s\in [t_1, t_2], |\slashed X(s)|\in I_0\cup I_1$.

Recall (\ref{oct22eqn1}). From the rough estimate of the electric   field (\ref{june2eqn71}) in   Lemma \ref{roughgeneral}, the following estimate holds for any $s\in [t_1, t_2],$
\[
 \frac{\slashed X(s )\cdot \slashed V(s  )}{\sqrt{1+|V(s )|^2}}-  \frac{\slashed X(t_1  )\cdot \slashed V(t_1  )}{\sqrt{1+|V(t_1 )|^2}}\geq \frac{4}{5}2^{2  \gamma_1 M_t-2\gamma_2 M_t} (s-t_1) - 2^{(\kappa_0-\gamma_2)M_t + 5M_t/3+10\epsilon M_t}(s-t_1),
\]
\[
\Longrightarrow  \frac{\slashed X(s )\cdot \slashed V(s  )}{\sqrt{1+|V(s )|^2}}  
\geq \frac{3}{4}2^{ 2  \gamma_1 M_t-2\gamma_2 M_t } (s-t_1)-\frac{11}{10} 2^{\kappa_0 M_t+   \gamma_1 M_t-\gamma_2 M_t},
\]
\be\label{nov28eqn78}
\Longrightarrow |X(s)|^2 -|X(t_1)|^2 \geq \frac{3}{4}2^{2 \gamma_1 M_t-2\gamma_2 M_t} (s-t_1)^2-\frac{11}{5} 2^{\kappa_0 M_t+   \gamma_1 M_t-\gamma_2 M_t } (s-t_1). 
\ee
Since $|\slashed X(t_2)|\leq 2^{\kappa_0+10}$, from the above estimate, we know that 
\[
2^{2\gamma_1 M_t-2\gamma_2 M_t} (t_2-t_1)^2\leq 2^{2\kappa_0 +20}, \quad \Longrightarrow |t_2-t_1|\leq 2^{\kappa_0+ \gamma_2 M_t-\gamma_1 M_t+10}.
\]
 
 From the above estimates, we know that,  if the lapse of time is greater than $2^{\kappa_0+\gamma_2   M_t -  \gamma_1 M_t+10}$, then the space characteristics will leave $I_0\cup I_1$, i.e., entering the region $\cup_{i=2}^K I_i\cup (I_2/I_1)$. Moreover,   from (\ref{oct22eqn1}) and  the rough estimate of the electric   field   (\ref{june2eqn71}) in   Lemma \ref{roughgeneral}, for any $s\in [t_1, t_2],$ we have
 \be 
  \big| |\slashed V(s)|^2 -|\slashed V(t_1)|^2 \big| +  \big| |  V(s)|^2 -|  V(t_1)|^2 \big|  \leq 2^{\gamma_2  M_t +2}\int_{t_1}^{t_2} \|E(s,x)\|_{L^\infty_x} ds \lesssim  2^{\kappa_0+2\gamma_2  M_t      }  2^{  (5/3+ 10\epsilon)M_t} \lesssim  1. 
 \ee

 \quad \textbf{Case} $2$:\quad  Given any $t_1, t_2\in  [\tau_{\ast}, \tau^{\ast}]$, if  $\forall s\in [t_1, t_2], |\slashed X(s)|\in I_i\cup I_{i+1} $ .

  Recall    (\ref{oct22eqn1}).   For any $s \in [t_1, t_2]$,  from the estimate (\ref{projest}) in  Lemma \ref{horprojlemma},  we have
  \[
  \frac{\slashed X(  s  )\cdot \slashed V( s )}{\sqrt{1+|V(s )|^2}}- \frac{\slashed X( t_1  )\cdot \slashed V( t_1 )}{\sqrt{1+|V(t_1 )|^2}}     \geq \frac{9}{10} 2^{ 2\gamma_1 M_t-2\gamma_2 M_t} (s -t_1)-   C \big[ \sum_{b\in \{0,2/3, 5/6,1\}}   2^{(1-b)a_i  } 2^{ b( \gamma_1  -\gamma_2   )M_t+ (\gamma_1-\gamma_2-2 \epsilon) M_t  }(t_2-t_1)
\]
\[
+ 2^{a_i+3\alpha^{\star} M_t/4 -\gamma_2 M_t+10\epsilon M_t}\big]
\]
\be\label{oct30eqn41}
\Longrightarrow    \frac{\slashed X(s   )\cdot \slashed V(s   )}{\sqrt{1+|V(s  )|^2}}  \geq \frac{4}{5}  2^{  2\gamma_1 M_t-2\gamma_2 M_t}   (s -t_1) -  \frac{11}{10}|\slashed  X(t_1)|2^{  \gamma_1 M_t-\gamma_2 M_t},
\ee
\[
\Longrightarrow  | \slashed X(s )|^2  -  | \slashed X(t_1)|^2   
   \geq  \frac{4}{5} 2^{   2(\gamma_1 - \gamma_2) M_t} (s -t_1)^2  -  \frac{22}{10 }|\slashed  X(t_1)|2^{    (\gamma_1 - \gamma_2) M_t}  (s -t_1), 
\]
\be\label{oct30eqn21}
\Longrightarrow | \slashed X( s  )| \geq 2^{ (\gamma_1 - \gamma_2) M_t -5}  (s -t_1)  -2|\slashed  X( t_1)|, \quad \Longrightarrow |t_2-t_1| \lesssim 2^{a_i -(\gamma_1 - \gamma_2) M_t}. 
\ee

From the above estimate of the time difference, the fact that $t_1, t_2\leq 2^{\epsilon M_t/100}$, and  the estimate (\ref{projest}) in  Lemma \ref{horprojlemma},    the following estimate holds for any $s \in [t_1, t_2]$,
\[
\big| | \slashed V(s )|  -  | \slashed V(t_1)|\big| \lesssim \big(  \sum_{b\in  \{0,2/3, 5/6,1\} }   2^{-ba_i } 2^{ b (\gamma_1 -\gamma_2) M_t   + (\gamma_1 -2 \epsilon) M_t}   \big)(s-t_1)  +  2^{   (3 {\alpha}^{\star} +\epsilon)M_t/4} \leq  2^{   \gamma_1 M_t-3\epsilon M_t/2 },
\]
\[
\big| |  V(s )|  -  |   V(t_1)|\big| \lesssim   \big[  \big(  \sum_{b\in  \{0,2/3, 5/6,1\} }   2^{-ba_i } 2^{ b (\gamma_1  -\gamma_2) M_t  +   (\gamma_1 -2 \epsilon) M_t }   \big)(s-t_1)  +  2^{  (3 {\alpha}^{\star}+20\epsilon)M_t/4} \big]\leq  2^{ \gamma_2 M_t -3\epsilon M_t/2  }   .
\]
 Moreover, if $\slashed X( t_1  )\cdot \slashed V( t_1 )\geq 0$, i.e., $d |X(t )|^2/dt\big|_{t=t_1}\geq 0$, and $|\slashed X( t_1  )|\in I_i/I_{i-1}=(2^{a_i+5}, 2^{a_i+10}]$, then we can improve the obtained estimate  (\ref{oct30eqn41}) as follows, 
 \[
 \textup{If $\slashed X( t_1  )\cdot \slashed V( t_1 )\geq 0$}, \quad  \frac{\slashed X(s   )\cdot \slashed V(s   )}{\sqrt{1+|V(s  )|^2}}  \geq \frac{4}{5}  2^{ 2(\gamma_1 -\gamma_2) M_t}   (s -t_1) -   2^{-10+a_i } 2^{ (\gamma_1 -\gamma_2) M_t},
 \]
 \be\label{oct30eqn51}
 \Longrightarrow \forall s\in [t_1, t_2], \quad   | \slashed X( s  )| \geq 2^{ (\gamma_1 -\gamma_2) M_t  -1}  (s -t_1)  +2^{-1}|\slashed  X( t_1)|
 \ee
 Recall (\ref{overlapintervals}).   From the above estimate, we  have  $| \slashed X( s  )|\in [2^{a_k+3}, 2^{a_k+8}]\subsetneq (2^{a_k}, 2^{a_{k}+10}]$, i.e., the space characteristics will leave $I_i$ and enter $I_{i+1}$ instead of $I_{i-1}$. 

 To sum up, our desired \textbf{Fact (i)} and \textbf{Fact (ii)} hold from the above discussion. Now, we are ready to improve the bootstrap assumption (\ref{bootstrap1}).  

 Thanks to \textbf{Fact (i)}, we only have to show that the trajectory of $|\slashed X(s)|, s\in  [\tau_{\ast}, \tau^{\ast}]$, visit intervals $I_i\cup I_{i+1}, i\in\{0, \cdots, K-1\}$ for at most     $C M_t$ number of times, which is only logarithmic.

Since the $d|\slashed X(s)|/ds$ is bounded, the trajectory of $|\slashed X(s)|$ can only visit intervals $J_i:=I_i/I_{i-1}, i\in\{0, \cdots, K\}$ for finite times. We order the visited intervals within $[\tau, \tau^{\star}]$ with respect to time as follows $(J_{i_1},J_{i_2},\cdots, J_{i_X})$, $\forall p\in\{1,\cdots, X\}, i_{p}\in \{0,\cdots, K\}$. Due to the possible revisit scenario, e.g., $J_1\longrightarrow J_2 \longrightarrow J_1\longrightarrow \cdots$, the size of  finite number  $X$ is not    clear. 

 Due to the continuity of characteristics, we have $ i_{p}-i_{p+1}\in\{1,-1\}$. Let  $ {i_{\iota_0}}$ be the first local minimum of the ordered set $(i_{1}, i_2, \cdots, i_{X})$, then $\iota_0\leq K+1, i_{\iota+1}= i_{\iota}+1$. Let $\tau_{ \iota_0}$ be the time such that the   $|\slashed X(\tau_{ \iota_0})|$leaves $I_{i_{ \iota_0}}$ and enters $I_{i_{ \iota_0 }+1}/I_{ \iota_0 }$. As $|\slashed X(s)|$ increases at time $\tau_{ \iota_0 }$, we know that  $d |\slashed X(t)|^2/dt\big|_{t=\tau_{ \iota_0}}\geq 0 $. Therefore, from the obtained \textbf{Fact (ii)}, we know that $i_{\kappa}\geq i_{ \iota_0 }$  for any $\kappa \geq  \iota_0$. Let $\iota_1:=\inf\{\kappa: i_{\kappa}\geq i_{\iota_0}+2,  \kappa\geq \iota_0\}$. We know that $J_{i_{\iota_1}}=I_{i_{\iota_0+2}}/I_{i_{\iota_0+1}}$ and $\cup_{\kappa\in [ {\iota_0},  {\iota_1}-1]\cap \Z} J_{i_{\kappa}} \subset I_{i_{\iota_0+1}}\cup  I_{ \iota_0 }$.

Similarly,  let $\tau_1$ be the time such that   $|\slashed X(\tau_{ \iota_1})|$leaves $I_{i_{ \iota_1}-1}=I_{i_{\iota_0+1}}$ and enters $J_{\iota_{\iota_1}}= I_{i_{ \iota_0 }+2}/I_{ i_{ \iota_0 }+1 }$, from the obtained \textbf{Fact (ii)}, we know that $i_{\kappa}\geq i_{ \iota_0 }+1$  for any $\kappa \geq  \iota_1$. Therefore, inductively, we can define a sequence of   $\{\iota_k\}_{k=0}^{m}$ such that the following property holds, 
 \begin{multline}\label{nov4eqn12}
 \{\kappa: i_{\kappa}\geq i_{\iota_m}+1\} =\emptyset,\quad \forall k\in\{0,\cdots, m-1\}, \quad  i_{\iota_{k+1}}\geq i_{\iota_k}+1,\quad  X_{\iota_k}:=\cup_{\kappa\in [\iota_k, \iota_{k+1}-1]\cap \Z} J_{i_{\kappa}}\subset I_{i_{\iota_k}+1}\cup I_{i_{\iota_k}},\\
  X_{\iota_m}:=\cup_{\kappa\in [\iota_m, K]\cap \Z} J_{i_{\kappa}}\subset I_{i_{\iota_m}+1}\cup I_{i_{\iota_m}}.
\end{multline}
If $i_{\iota_m}=K$, we use the convention that $I_{K+1}:=\emptyset.$

Since    $\{i_{\iota_k}\}_{k=0}^{m}$  is an increasing sequence with upper bound $K$, we know that $m\leq K+1$.  From the above discussion, we regroup the ordered set $(J_{i_1},J_{i_2},\cdots, J_{i_X})$, as follows, 
\be\label{nov4eqn11}
(J_{i_1}, J_{i_2},\cdots, J_{i_{\iota_0-1}}, X_{\iota_0},X_{\iota_1},\quad X_{\iota_m}),
\ee
in which the order is still consistent with the time order that the characteristic  $|\slashed X(s)|, s\in [\tau_{\ast}, \tau^{\ast}]$  travels. 

Since the total number of sets in (\ref{nov4eqn11}) is less than $2(K+1)\leq 100 M_t$, see (\ref{overlapintervals}),  which is only logarithmic, from the relation in (\ref{nov4eqn12}) and  \textbf{Fact (i)}, our bootstrap assumption is improved.  Hence finishing the bootstrap argument, i.e., $\tau^{\ast}= \tau^{\star}$, which further implies that  $\tau^{\ast}= \tau^{\star}=\tau.$

  \textbf{Step} $\mathbf{2}$ \qquad Estimate of the full velocity characteristics within the time interval $[0, \tau]$. 

As a result of the  \textbf{Step} $\mathbf{1}$  and (\ref{2021dec18eqn1}), we know that  $ \alpha_\tau M_\tau \leq \alpha_t(\tau) M_t \leq (\alpha^{\star}-2 \epsilon)M_t$.  Now, our goal is to show that $\tau = t  .$ Recall   (\ref{oct22eqn1}). For any $(x,v)\in R_t(0)$,  $t_1, t_2\in [t^{\ast}, \tau]$, the following estimate holds from  the estimate (\ref{2022jan18eqn51}) in Lemma \ref{bootstraplemma2} 
\[
\big||V(t_2)|^2  - |V(t_1)|^2   \big|\leq  2\big|\int_{t_1}^{t_2}    {V}(s)\cdot E(s, X(s) ) d s\big| \lesssim  2^{   (2\gamma-5\epsilon) M_t   }\leq 2^{  2\gamma M_t-4.5\epsilon M_t}.
\]
 Therefore, our bootstrap assumption in (\ref{2021dec18eqn22}) is improved.  Hence finishing the bootstrap argument, i.e., $\tau = t.$  Recall the definition of the majority set in (\ref{may9en21}). Since $\tau = t$, after rerunning the argument in  \textbf{Step} $\mathbf{1}$,   we have 
 \[
\beta_t = \sup_{s\in[0, t]}\beta_t(s) \leq 1-2\epsilon, \quad \alpha_t = \sup_{s\in[0, t]}\alpha_t(s) \leq \alpha^{\star}-  \epsilon. 
 \]
 Hence finishing the proof of our desired estimate (\ref{sizeofmajority}). 
\end{proof}

The key ingredient of the first induction in (\ref{logic}) is the following Lemma, in which we  use the smoothing effect by localizing the frequency on Fourier side and then doing normal form transformation. It worth to remark that the following Lemma is also valid for the RVP system (\ref{vlasovpo}) because we will not use the extra equation in the sRVM system (\ref{sRVM}) in the proof. 
\begin{lemma}\label{horprojlemma}
Let $t\in [0, T), a_p\in \Z, \alpha^{\star}=7/10+\iota, \iota:=10^{-2}$,  $t_1, t_2\in [0, t]$ be fixed, s.t., $M_t\gg 1, $ $\forall s\in [t_1, t_2], |\slashed X(s)|\sim 2^{a_p }, {\alpha}_s M_s\leq  \alpha^{\star} M_t,  |\slashed V(s)|\sim 2^{\gamma_1 M_t} ,  |  V(s)|\sim  2^{\gamma_2 M_t} $,  where  $\gamma_1\in  [ \alpha^{\star} -2\epsilon , \alpha^{\star}]  $, $\gamma_2 \leq (1-3\epsilon)  $. Assume that $C(\cdot, \cdot):\R^2 \times \R^3\rightarrow \R^3$ is a smooth function s.t., the following estimate holds for any $  s \in [t_1, t_2]$,
\be\label{assumptiononcoefficients}
 \big|C(\slashed x, V(s))\big| +|\slashed X(s)|\big|\nabla_{\slashed x}C(\slashed x, V(s))|_{\slashed x = \slashed X(s)} \big| + \sum_{|\alpha|\leq 10} |V(s)|^{|\alpha|} \big|\nabla_{v}^{\alpha}C(\slashed x, v)|_{v = V(s)} \big|\lesssim \mathcal{M}(C).
\ee
 Then the following estimate holds, 
\be\label{projest}
\big|\int_{t_1}^{t_2} C(\slashed X(s), V(s)) \cdot  E(s, X(s) ) d s \big|\lesssim  \mathcal{M}(C) \big[\big(\sum_{b\in  \{0,2/3, 5/6,1\}  }   2^{-ba_p  } 2^{ b(\gamma_1-\gamma_2)M_t+ (\gamma_1-2\epsilon )  M_t}\big)(t_2-t_1)+  2^{(3\alpha^{\star} + \epsilon)M_t/4}\big]. 
\ee
\end{lemma}
\begin{proof}
See section \ref{mainlemma1}.  
\end{proof}

The key ingredient of the second induction in (\ref{logic}) is the following Lemma, in which we  use both the smoothing effect and the hyperbolic natural of the sRVM system (\ref{sRVM}) by using the extra equation. 

\begin{lemma}\label{bootstraplemma2}
Let $t\in [0, T), a_p\in \Z,  \alpha^{\star}=7/10+\iota, \iota:=10^{-2}$,  $ \gamma \in[  1-3\epsilon, 1]$,   $t_1, t_2\in [0, t]$ be fixed, s.t., $M_t\gg 1, $  $\forall s\in [t_1, t_2], |\slashed X(s)|\sim 2^{a_p }, |  V(s)|\sim 2^{\gamma  M_t} ,  \alpha_s M_s \leq \alpha^{\star}M_t $.  Then the following estimate holds, 
\be\label{2022jan18eqn51}
\big|\int_{t_1}^{t_2} {V}(s)  \cdot   E(s, X(s) ) d s \big|\lesssim     2^{   (2\gamma-5\epsilon) M_t }  .
\ee
\end{lemma}
\begin{proof}
See section \ref{mainlemma2}.  
\end{proof}

\section{Proof of Lemma \ref{horprojlemma}}\label{mainlemma1}
 
We first rule out the case $|\slashed X(s)|$ is very small, If $a_p\leq 2(\gamma_1-\gamma_2)M_t-20\epsilon M_t$, then from the estimate (\ref{june2eqn71}) in Lemma \ref{roughgeneral}, we have 
\[
 \big|\int_{t_1}^{t_2} C(\slashed X(s), V(s)) \cdot  E(s, X(s) ) d s \big|\lesssim 2^{\alpha^{\star}M_t + 2\epsilon M_t} 2^{-a_p/2} \mathcal{M}(C) (t_2-t_1)\lesssim 2^{-a_p +(\gamma_1-\gamma_2)M_t}2^{(\alpha^{\star}-5\epsilon )M_t}\mathcal{M}(C) (t_2-t_1).
\]

\textbf{Now, it would be sufficient to consider the case $a_p\geq 2(\gamma_1-\gamma_2)M_t-20\epsilon M_t$}. After doing dyadic decomposition for the size of frequency, the size of $\slashed v$, and $v$, we have
\begin{multline}
\big|\int_{t_1}^{t_2} C(\slashed X(s), V(s)) \cdot  E(s, X(s) ) d s \big|\lesssim \sum_{k,j_2\in \Z_+, j_1\in[0, j_2+2]\cap \Z} \big|H_{k;j_1,j_2}(t_1, t_2)\big|,\\  
H_{k;j_1,j_2}(t_1, t_2)= \int_{t_1}^{t_2} C(\slashed X(s), V(s)) \cdot  E_{k;j_1,j_2}(s, X(s) ) d s,
\end{multline}
where $E_{k;j_1,j_2}(\cdot, \cdot)$ is defined in (\ref{june9eqn2}). From the obtained estimates (\ref{2022jan18eqn30}) and (\ref{2022jan18eqn31}), we know that it would be sufficient to consider the case fixed $k,j_1, j_2\in \Z_+$, s.t.,  $k\in [0,3M_t + 5\epsilon M_t]\cap \Z$,  $j_2\in [0, (1+\epsilon)M_t]\cap \Z $ and $j_1\in [0, j_2+2]\cap \Z$.  Moreover, based on the possible sizes of $k, j_1, j_2$, we split into three sub-cases as follows. 

$\bullet$ If $k\leq j_2 + (\gamma_1-\gamma_2 + \alpha^{\star}-5\epsilon)M_t$ or $k\geq 2j_1+j_2-(\alpha^{\star}-5\epsilon)M_t.$

From the estimate (\ref{may31eqn1}) in Lemma \ref{roughesti}, we have 
\be\label{2022jan18eqn21}
\big|H_{k;j_1,j_2}(t_1, t_2)\big|\lesssim \min\{2^{-k+j_2+2j_1},  2^{k-j_2+\epsilon M_t} 2^{-a_p}\} \mathcal{M}(C)(t_2-t_1)\lesssim \sum_{b\in\{0,1\}} 2^{-b a_p} 2^{b(\gamma_1-\gamma_2)M_t +(\alpha^{\star}-5\epsilon)M_t}. 
\ee

$\bullet$ If $k\in [ j_2 + (\gamma_1-\gamma_2 + \alpha^{\star}-5\epsilon)M_t,2j_1+j_2-(\alpha^{\star}-5\epsilon)M_t ]$, $j_1\leq (\gamma_1-\gamma_2)M_t/2 +(\alpha^{\star}-6\epsilon)M_t.  $

From the estimate (\ref{may31eqn2}) in Lemma \ref{roughesti}, we have
\be\label{2022jan18eqn22}
\big|H_{k;j_1,j_2}(t_1, t_2)\big| \lesssim 2^{-a_p/2 +j_1 +\epsilon M_t } \mathcal{M}(C)(t_2-t_1) \lesssim 2^{-a_p/2} 2^{(\gamma_1-\gamma_2)M_t/2 +(\alpha^{\star}-5\epsilon)M_t}\mathcal{M}(C)(t_2-t_1) . 
\ee

$\bullet$ If $k\in [ j_2 + (\gamma_1-\gamma_2 + \alpha^{\star}-5\epsilon)M_t,2j_1+j_2-(\alpha^{\star}-5\epsilon)M_t ]$ and  $j_1\geq (\gamma_1-\gamma_2)M_t/2 +(\alpha^{\star}-6\epsilon)M_t $.

Let $\theta_{V(s)}(v):=(\hat{V}(s)-\hat{v})/|\hat{V}(s)-\hat{v}|$. Based on the possible size of  $\theta_{V(s)}(v)\cdot\xi/|\xi| $,  $|\hat{V}(s)-\hat{v}|$, and the angle between   $ (-X_2(s), X_1(s),0)/|\slashed X(s)|$ and $ \theta_{V(s)}(v)$,  we decompose  $H_{k;j_1,j_2}(t_1, t_2)$ into three parts as follows, 
\be 
H_{k;j_1,j_2}(t_1, t_2)= \sum_{i=1,2,3} \int_{t_1}^{t_2}   H_{k,j_1,j_2}^i(s ) d s,  
\ee
where
\be\label{2022jan16eqn1}
 H_{k,j_1,j_2}^1(s )=  \int_{\R^3} \int_{\R^3 } e^{i X(s)\cdot \xi } i  C(\slashed X(s), V(s)) \cdot  \xi |\xi|^{-2}  \widehat{f}(s ,\xi ,v)\varphi_{k;j_1,j_2}(v, \xi)  \varphi_{l_1,\alpha}(v, X(s), \tilde{V}(s))d v d \xi ,
\ee
\be\label{2022jan16eqn2}
   H_{k,j_1,j_2}^2(s )=  \int_{\R^3} \int_{\R^3 } e^{i X(s)\cdot \xi }   i  C(\slashed X(s), V(s)) \cdot   \xi |\xi|^{-2} \widehat{f}(s ,\xi ,v)\varphi_{k;j_1,j_2}(v, \xi)  \psi_{\leq l_2 }(  \theta_{V(s)}(v)\cdot \tilde{\xi} ) \big(1- \varphi_{l_1,\alpha}(v, X(t), \tilde{V}(s))\big)  d v d \xi ,
  \ee
\be\label{2022jan16eqn3}
 H_{k,j_1,j_2}^3(s )=  \int_{\R^3} \int_{\R^3 } e^{i X(s)\cdot \xi }   i   C(\slashed X(s), V(s)) \cdot  \xi |\xi|^{-2} \widehat{f}(s ,\xi ,v) \varphi_{k;j_1,j_2}(v, \xi)   \psi_{> l_2 }(  \theta_{V(s)}(v)\cdot \tilde{\xi} ) \big(1- \varphi_{l_1,\alpha}(v, X(s), \tilde{V}(s))\big)  d v d \xi,
  \ee
where $\varphi_{k;j_1,j_2}(v, \xi):=  \varphi_{j_1}(\slashed v)\varphi_{j_2}(v)  \varphi_k(\xi)$ the cutoff function $ \varphi_{l_1,\alpha}(v, \tilde{V}(t))$ is defined as follows, 
 \be\label{cutofffunc}
 \varphi_{l_1,\alpha}(v, X(s), \tilde{V}(s)):= \psi_{\leq l_1  }(|\tilde{v}-\tilde{V}(s)|) + \psi_{>l_1  }(|\tilde{v}-\tilde{V}(s)|)\psi_{\leq \alpha} (\theta_{V(s)}(v)\times (-X_2(s), X_1(s),0)/|\slashed X(s)|), 
 \ee
 and the thresholds of cutoff functions are chosen as follows, 
 \be\label{2022jan18eqn33}
l_1:= (\alpha^{\star} M_t+(\gamma_1-\gamma_2)M_t-a_p +k-3j_2)/2- 5\epsilon M_t, \quad \alpha = \frac{2}{3}l_1, \quad l_2:= (\gamma_1-\gamma_2+\alpha^{\star}-10\epsilon)M_t+j_2-k+\alpha. 
\ee

$\oplus$\qquad The estimate of $ H_{k,j_1,j_2}^1 (s ).$

Recall (\ref{2022jan16eqn1}).  As in \cite{wang2}[Lemma 4.3], for any fixed  $a, b\in \mathbb{S}^2, a\neq b,$ $0< \epsilon \ll 1$, we have
 \be\label{april29eqn1}
S_{a,b}:=\{c: c\in \mathbb{S}^2, |(a-c)\times b|\leq \epsilon\}, \quad |S_{a,b}|\lesssim \epsilon^{3/2}, \quad |a-c|\lesssim \epsilon^{1/2}.
\ee  
From  the above estimate and  the volume of support of $v$, we have
\be\label{june2eqn61}
\big|\int_{t_1}^{t_2}   H_{k,j_1,j_2}^1(s ) d s\big| \lesssim 2^{\epsilon M_t} \min\{ 2^{-k+3j_2+ 2l_1} + 2^{-k+3j_2+ 3\alpha}  \}\lesssim 2^{-a_p + (\gamma_1-\gamma_2)M_t} 2^{(\alpha^{\star}-6\epsilon)M_t}. 
\ee

$\oplus$\qquad The estimate of $ H_{k,j_1,j_2}^2 (s ).$

Note that, in terms of kernel, we have
 \be\label{april28eqn11}
    H_{k,j_1,j_2}^2(s )   = \int_{\R^3} \int_{\R^3} f(s,X(s)-y, v) \tilde{K}_{k,l_2}(y, v,V(s))  \big(1- \varphi_{l_1,\alpha}(v, X(s), \tilde{V}(s))\big) \varphi_{j_1}(\slashed v)\varphi_{j_2}(v) d v d y, 
 \ee
 where
 \[
 \tilde{K}_{k,l_2}(y, v,V(s)) = \int_{\R^3 } e^{i y\cdot \xi} i \xi |\xi|^{-2} \varphi_k(\xi)   \psi_{\leq l_2 }(  \theta_{V(s)}(v)\cdot \tilde{\xi} )  d \xi.  
 \]
 By doing integration by parts in $\theta_{V(s)}(v)$ direction and directions perpendicular to  $(\theta_{V(s)}(v))^{\bot}$,  we have the following estimate  for the kernel $ \tilde{K}_{k,l_2}(y, v,V(s))$,
 \be\label{april28eqn1}
 | \tilde{K}_{k,l_2}(y, v,V(s))| \lesssim 2^{2k+ l_2}(1+ 2^{k+l_2}|y\cdot \theta_{V(s)}(v)|)^{-N_0^3} (1+2^k|y\times \theta_{V(s)}(v)|)^{-N_0^3}.
 \ee
 From the above estimate, we know that ``$y$'' is localized inside a cylinder with base in the plane perpendicular to $\theta_{V(s)}(v)^{}$. Due to the cutoff function $ (1- \varphi_{l_1,\alpha}(v, X(s), \tilde{V}(s)) )  $ in (\ref{april28eqn11}),  the angle between $\theta_{V(s)}(v)$ and $ (-X_2(s), X_1(s),0)/|\slashed X(s)|)$ is greater than $2^{\alpha}$, which means that the intersection of the cylinder with any $x_1x_2$ plane  is less than $(2^{-k-\alpha})^2.$ 

Note that $a_p \geq -k-l_2+10 \epsilon M_t.$ Hence, from the cylindrical symmetry of solution and the above estimate, we have
\[
  \int_{t_1}^{t_2} \big| H_{k,j_1,j_2}^2(s)\big|  ds   \lesssim  2^{2k+l_2}   \int_{t_1}^{t_2}  \int_{\R^3} \int_{\R^3} (1+ 2^{k+l_2}|y\cdot \theta_{V(s)}(v)|)^{-N_0^3} (1+2^k|y\times \theta_{V(s)}(v)|)^{-N_0^3} f(s, X(s)-y, v) \varphi_{j_1}(\slashed v)
 \]
\[ 
 \times\varphi_{j_2}(v) \big(1- \varphi_{l_1,\alpha}(v, X(s), \tilde{V}(s))\big)  d v d y  ds  \lesssim 1 + \int_{t_1}^{t_2}  \int_{\R^3} \int_{\R^3}  \frac{2^{2k+l_2-k-\alpha +\epsilon M_t/10}}{|\slashed X(s)- \slashed y|} f(s, X(s)-y, v)\varphi_{j_1}(\slashed v)\varphi_{j_2}(v) d v d y   ds
\] 
 \be\label{june2eqn11}
  \lesssim 1+ 2^{2k+l_2-k-\alpha +\epsilon M_t/10}2^{-j_2 } 2^{-a_p }\lesssim 2^{-a_p + (\gamma_1-\gamma_2+\alpha^{\star}-6\epsilon)M_t }. 
\ee

$\oplus$\qquad The estimate of $ H_{k,j_1,j_2}^3 (s ).$

Recall (\ref{2022jan16eqn3}).   Note that, for any $(v,\xi)\in supp( \psi_{> l_2 }(  \theta_{V(s)}(v)\cdot \tilde{\xi} ) (1- \varphi_{l_1,\alpha}(v, X(s), \tilde{V}(s))) )$, c.f., (\ref{cutofffunc}), we have
 \be\label{april28eqn13}
|\hat{V}(s)-\hat{v}|\gtrsim |\tilde{V}(s)-\tilde{v}|-  2^{-2M_t + 4\epsilon M_t}\gtrsim 2^{l_1}, \quad \Longrightarrow 
|\hat{V}(t)\cdot\tilde{\xi} - \hat{v}\cdot \tilde{\xi}|= |\hat{V}(t)-\hat{v}||\theta_{V(t)}(v)\cdot \tilde{\xi}|\gtrsim 2^{l_1+l_2}. 
\ee

Let $g(s,x,v):= f(s, x+ s\hat{v}, v).$ For $H_{k,j_1,j_2}^3(s)$, we do integration by parts in $s$ once. As a result, we have
\be\label{june2eqn54}
  \big|\int_{t_1}^{t_2} H_{k,j_1,j_2}^3(s)d s \big|\lesssim  \big| End_{k,j_1,j_2}(t_1, t_2) \big| +  \big|\widetilde{H}_{k,j_1,j_2}^1(t_1,t_2) \big| +   \big|\widetilde{H}_{k,j_1,j_2}^2(t_1,t_2) \big|,
\ee
where
\[
 End_{k,j_1,j_2}(t_1, t_2):= \sum_{a=1,2} \int_{\R^3} \int_{\R^3 } e^{i X(t_a)\cdot \xi- i  t_a \hat{v}\cdot \xi}   C(\slashed X(t_a), V(t_a))\cdot \xi |\xi|^{-2}  \hat{g}(t_a ,\xi ,v)     \varphi_{k;j_1,j_2}(v,\xi)     (\hat{V}(t_a)\cdot\xi - \hat{v}\cdot \xi )^{-1} 
\]
\[
\times   \psi_{> l_2 }(  \theta_{V(t_a)}(v)\cdot \tilde{\xi} )  \big(1-  \varphi_{l_1,\alpha}(v, X(t_a), \tilde{V}(t_a))\big)    d v d \xi
\]
 \be\label{april28eqn14}
=  \sum_{a=1,2}  \sum_{ n\in[l_1,2]\cap \mathbb{Z}}    \sum_{ l\in[l_2,2]\cap \mathbb{Z}} \int_{\R^3} \int_{\R^3} f(t_a, X(t_a)-y, v) \widetilde{K}^0_{k,n,l }(y, X(t_a),  V(t_a), v)  \varphi_{j_1}(\slashed v)\varphi_{j_2}(v) d y d v, 
\ee
\[
 \widetilde{H}_{k,j_1,j_2}^1(t_1, t_2):= \int_{t_1}^{t_2}  \int_{\R^3} \int_{\R^3 } e^{i X(s)\cdot \xi- i s \hat{v}\cdot \xi}     C(\slashed X(s), V(s)) \cdot \xi |\xi|^{-2} \varphi_{k,j_1,j_2}(v, \xi )  \p_s  \hat{g}(s ,\xi ,v)   \big(1-  \varphi_{l_1,\alpha}(v, X(s ), \tilde{V}( s))\big) 
  \]
 \be\label{april28eqn15}
  \times   (\hat{V}(s)\cdot\xi - \hat{v}\cdot \xi )^{-1}   \psi_{> l_2 }(  \theta_{V(s)}(v)\cdot \tilde{\xi} ) d v d \xi d s,  
\ee
\[
 \widetilde{H}_{k,j_1,j_2}^2(t_1, t_2):=\int_{t_1}^{t_2} \int_{\R^3} \int_{\R^3 } e^{i X(s)\cdot \xi }  \p_s \big(   C(\slashed X(s), V(s))\cdot  \xi |\xi|^{-2}  \psi_{> l_2 }(  \theta_{V(s)}(v)\cdot \tilde{\xi} )(\hat{V}(s)\cdot\xi - \hat{v}\cdot \xi )^{-1} \big(1-  \varphi_{l_1,\alpha}(v, X(s), \tilde{V}(s))\big) \big) 
      \]
  \[
  \times   \varphi_{k;j_1,j_2}(v,\xi)   \varphi_{j_1}(\slashed v)\varphi_{j_2}(v) \hat{f}(s ,\xi ,v) d v d \xi d s 
\]
 \be\label{april28eqn16}
  =  \sum_{ n\in[l_1,2]\cap \mathbb{Z}}    \sum_{ l\in[l_2,2]\cap \mathbb{Z}} \int_{s_1}^{s_2} \int_{\R^3} \int_{\R^3} f(s, X(s)-y, v) \widetilde{K}^1_{k,n,l}(y,  X(s),  V(s), v)   \varphi_{j_1}(\slashed v)\varphi_{j_2}(v) d y d v d s,
\ee
where the kernels $\widetilde{K}^i_{k,l_1,l_2}(y, \slashed X(s), V(s), v), i\in\{0,1\},$ are defined as follow,
\be\label{april28eqn21}
\widetilde{K}^0_{k,n,l   }(y,   X(s),  V(s), v):= \int_{\R^3} e^{i y \cdot \xi} \frac{ \varphi_k(\xi)  C(\slashed X(s), V(s))\cdot  \xi}{ |\xi|^{2} (\hat{V}(s)\cdot\xi - \hat{v}\cdot \xi )}       \psi_{  l  }(  \theta_{V(t_a)}(v)\cdot \tilde{\xi} )  \psi_{n}(|\tilde{v}-\tilde{V}(t_a)|)     \big(1-  \varphi_{l_1,\alpha}(v, X(s ), \tilde{V}( s))\big)d \xi, 
\ee
\[
\widetilde{K}^1_{k, n,l  }(y,  X(s), V(s), v):= \int_{\R^3} e^{i y \cdot \xi} \p_s \big( \frac{ \varphi_k(\xi)  C(\slashed X(s), V(s))\cdot  \xi}{ |\xi|^{2} (\hat{V}(s)\cdot\xi - \hat{v}\cdot \xi )}      \psi_{> l_2 }(  \theta_{V(s)}(v)\cdot \tilde{\xi} )
\]
\be\label{april28eqn22}
   \times \psi_{  l  }(  \theta_{V(s)}(v)\cdot \tilde{\xi} )       \psi_{n}(|\tilde{v}-\tilde{V}(s)|)  \big(1-  \varphi_{l_1,\alpha}(v, X(s), \tilde{V}(s))\big) \big)  d \xi.
\ee

By using the same strategy as in the obtained estimate (\ref{june2eqn11}), we have 
\[
\big|  End_{k,j_1,j_2}(t_1, t_2)\big| \lesssim \sum_{a=1,2} \sum_{ n\in[l_1,2]\cap \mathbb{Z}}    \sum_{ l\in[l_2,2]\cap \mathbb{Z}} \int_{\R^3} \int_{\R^3} f(s_a, X(t_a)-y, v)  2^{k -n}   \varphi_{j_1}(\slashed v)\varphi_{j_2}(v)  (1+2^{k+l}|y\cdot\theta_{V(t_a)}(v) |)^{- N_0^3} \] 
\[
\times  (1+2^{k }|y\times \theta_{V(t_a )}(v) |)^{- N_0^3 }   \mathcal{M}(C)  d y d v \lesssim 2^{5\epsilon M_t} 2^{-l_1-\alpha - j_2 -a_p  } \mathcal{M}(C) \lesssim 2^{-a_p/6+2j_2/3-5(\alpha^{\star}+\gamma_1-\gamma_2)M_t/3}\mathcal{M}(C)
\]
\be\label{2022jan16eqn20}
 \lesssim 2^{-a_p/6+7M_t/3-10 \alpha^{\star} M_t/3+50\epsilon M_t}\mathcal{M}(C) \lesssim 2^{  \alpha^{\star} M_t/2} \mathcal{M}(C).
\ee
 
\noindent $\dagger$\quad The estimate of $  \widetilde{H}_{k,j_1,j_2}^1 (t_1,t_2)$.

Recall (\ref{april28eqn15}) and  (\ref{sRVM}). We have
\[
\p_s g(s,x,v) = \nabla_x\phi(s,x+s\hat{v})\cdot \nabla_v f(s,x+s\hat{v}, v ). 
\]
Hence, after doing integration by parts in $v$, we have 
 \be\label{april28eqn17}
\widetilde{H}_{k,j_1,j_2}^1 (t_1,t_2):=  \sum_{ n\in[l_1,2]\cap \mathbb{Z}}    \sum_{ l\in[l_2,2]\cap \mathbb{Z}}   \int_{t_1}^{t_2} \int_{\R^3} \int_{\R^3 }\int_{\R^3}    \widetilde{K}^{j_1,j_2}_{k,n,l}(y,\slashed X(s), V(s), v)\cdot E (s, X(s)-y) f(s,X(s)-y,v)  dy d v d s,
\ee
where the kernel  $    \widetilde{K}^{j_1,j_2}_{k,n,l}(y, \slashed X(s),V(s), v)$ is defined as follows, 
 \[
   \widetilde{K}^{j_1,j_2}_{k,n,l}(y, \slashed X(s), V(s), v):= \int_{\R^3} e^{i x\cdot \xi }   \nabla_v \big( \varphi_{j_1}(\slashed v)  \varphi_{j_2}(v) \big(1-  \varphi_{l_1,\alpha}(v, X(s ), \tilde{V}( s)) (\hat{V}(s)\cdot\xi - \hat{v}\cdot \xi )^{-1}  \psi_{l }(  \theta_{V(s)}(v)\cdot \tilde{\xi} )
 \]
 \be\label{april29eqn11}
  \times     \psi_{ n }(|\tilde{v}-\tilde{V}(s)|) \big)  C(\slashed X(s), V(s))\cdot   \xi |\xi|^{-2} \varphi_k(\xi)   d \xi.   
\ee

As a result of direct computation, by   doing integration by parts in $\theta_{V(s)}(v)$ direction and directions perpendicular to  $(\theta_{V(s)}(v))^{\bot}$,  we have
 \be\label{may5eqn51}
|   \widetilde{K}^{j_1,j_2}_{k,n,l}(y, \slashed X(s), V(s), v)| \lesssim 2^{k+ l}\big( 2^{ -2l-2n-j_2 } + 2^{ -l-n-j_1 } \big)   (1+2^{k+l}|y\cdot\theta_{V(s)}(v) |)^{- N_0^3 } (1+2^{k }|y\times \theta_{V(s)}(v) |)^{-  N_0^3 } \mathcal{M}(C)  .
\ee

From the above estimate of kernel in (\ref{may5eqn51}) and the rough estimate of the electric field (\ref{june2eqn71}) in Lemma \ref{roughgeneral}, the following estimate holds after using the same strategy as in the obtained estimate (\ref{june2eqn11}), \[
\big|\widetilde{H}_{k,j_1,j_2}^1 (t_1,t_2)\big|\lesssim \sum_{ n\in[l_1,2]\cap \mathbb{Z}}    \sum_{ l\in[l_2,2]\cap \mathbb{Z}}   \int_{t_1}^{t_2}  \mathcal{M}(C)   2^{k+ l}\big( 2^{ -2l-2n-j_2 } + 2^{ -l-n-j_1 } \big) \|  \nabla_x\phi (s, \cdot ) \|_{L^\infty_x } 
\]
\[
 \times \varphi_{j_1}(\slashed v)\varphi_{j_2}(v)  (1+2^{k+l}|y\cdot\theta_{V(s)}(v) |)^{- N_0^3 } (1+2^{k }|y\times \theta_{V(s)}(v) |)^{-  N_0^3 }  f(s,X(s)-y,v)  dy d v d s
 \]
\[
\lesssim  (2^{-2l_1-l_2 -j_2 } + 2^{-l_1-j_1})2^{-\alpha -j_2 -a_p} 2^{\alpha^{\star}M_t -a_p/2+10\epsilon M_t}   \mathcal{M}(C)(t_2-t_1)
\]
\be\label{2022jan16eqn21}
\lesssim 2^{30\epsilon M_t}\big(2^{14M_t/3-17\alpha^{\star}M_t/3} + 2^{-2a_p/3 + 7M_t/3-10\alpha^{\star}M_t/3-(\gamma_1-\gamma_2)M_t/2 }\big) \mathcal{M}(C)(t_2-t_1).
\ee

\noindent $\dagger$\quad The estimate of $  \widetilde{H}_{k,j_1,j_2}^2 (s_1,s_2)$.

Recall (\ref{april28eqn16}) and the definition of kernel $\widetilde{K}^1_{k, n,l  }(y,  X(s), V(s), v)$ in (\ref{april28eqn22}). As a result of direct computations, the following estimate holds for the kernel after doing integration by parts in $\xi$ in $\theta_{V(s)}(v)$ direction and $(\theta_{V(s )}(v))^{\bot}$ directions, 
\[ 
\big| \widetilde{K}^1_{k, n,l  }(y,  X(s), V(s), v)\big| \lesssim 2^{k+l-l-n} \big( 2^{-l-n-\gamma_2 M_t} |\nabla_x\phi(s, X(s)) | + 2^{-\alpha + (\gamma_1-\gamma_2)M_t-a_p }   \big) 
\]
\be\label{2022jan16eqn11}
 \times  (1+2^{k+l}|y\cdot\theta_{V(s)}(v) |)^{- N_0^3 } (1+2^{k }|y\times \theta_{V(s)}(v) |)^{-  N_0^3 }\mathcal{M}(C). 
\ee
 From the above estimate of kernel and the rough estimate of the electric field (\ref{june2eqn71}) in Lemma \ref{roughgeneral}, we have 
\[
\big|\widetilde{H}_{k,j_1,j_2}^2 (s_1,s_2)\big|\lesssim \sum_{ n\in[l_1,2]\cap \mathbb{Z}}    \sum_{ l\in[l_2,2]\cap \mathbb{Z}} \int_{t_1}^{t_2}  \mathcal{M}(C) 2^{k-n}  \big( 2^{-l-n-\gamma_2 M_t} \|\nabla_x\phi(s, \cdot)\|_{L^\infty_x} + 2^{-\alpha + (\gamma_1-\gamma_2)M_t-a_p }   \big)2^{-k-\alpha- j_2 -a_p +10\epsilon M_t } d s
\] 
 \[
\lesssim \big(2^{-2l_1-l_2+(\alpha^{\star}-\gamma_2)M_t-a_p/2} + 2^{-5l_1/3 + (\gamma_1-\gamma_2)M_t-a_p  }\big)   2^{-2l_1/3-j_2-a_p+15\epsilon M_t} \mathcal{M}(C)(t_2-t_1)
 \]
\[
 \lesssim 2^{20\epsilon M_t}\big( 2^{-2k/3+3j_2 -8\alpha^{\star} M_t/3 -5(\gamma_1-\gamma_2)M_t/3 +a_p/6}+ 2^{-7k/6 + 5j_2/2-7\alpha^{\star}M_t/6-(\gamma_1-\gamma_2)M_t/6-5a_p/6 }\big) \mathcal{M}(C)(t_2-t_1)
\]
\be\label{2022jan16eqn22}
\lesssim 2^{30\epsilon M_t}\big(  2^{14 M_t/3-17\alpha^{\star} M_t/3  +a_p/6} +  2^{8M_t/3-11\alpha^{\star}M_t/3 -5a_p/6}  \big)\mathcal{M}(C)(t_2-t_1).
\ee
Recall (\ref{june2eqn54}). To sum up, after combining the obtained estimates (\ref{2022jan16eqn20}),  (\ref{2022jan16eqn21}), and  (\ref{2022jan16eqn22}), we have 
\[
\big|   \int_{t_1}^{t_2} H_{k,j_1,j_2}^3(s)d s\big|\lesssim \mathcal{M}(C) \big( \sum_{b\in \{0,2/3, 5/6\} } 2^{-b a_p +b(\gamma_1-\gamma_2) M_t +(\alpha^{\star}-10\epsilon)M_t}(t_2-t_1) \big) + 2^{\alpha^{\star} M_t/2}\mathcal{M}(C). 
\]
After combining the above estimate, the estimate (\ref{june2eqn61}), and (\ref{june2eqn11}), we have
\[
\big| H_{k;j_1,j_2}(t_1, t_2)\big|\lesssim \mathcal{M}(C) \big( \sum_{b\in \{0,2/3, 5/6,1\} } 2^{-b a_p +b(\gamma_1-\gamma_2) M_t +(\alpha^{\star}-6\epsilon)M_t}(t_2-t_1) \big) + 2^{ \alpha^{\star} M_t/2}\mathcal{M}(C).
\]
Hence finishing the proof of our desired estimate (\ref{projest}).

\section{Proof of Lemma \ref{bootstraplemma2}}\label{mainlemma2}
 
From the obtained estimates (\ref{may31eqn1}), (\ref{2022jan18eqn30}) and (\ref{2022jan18eqn31}), we know that it would be sufficient to consider the case fixed   $k\in [4M_t/5-20\epsilon M_t,3M_t + 5\epsilon M_t]\cap \Z$.   Moreover, from the obtained  estimates (\ref{2022jan18eqn71}) and (\ref{2022jan18eqn30}), we have 
\be\label{2022jan20eqn71}
\int_{t_1}^{t_2} \big|  {{\slashed V(s )}}  \cdot     P_k(\slashed{E}) (s,X(s)) \big| ds  \lesssim 2^{(1+10\epsilon )M_t} 2^{\alpha^{\star}M_t +10\epsilon M_t} \lesssim 2^{9M_t/5}. 
\ee
Hence, to control the LHS of (\ref{2022jan18eqn51}),  it would be sufficient to control the following term for     fixed   $k\in  [4M_t/5-20\epsilon M_t,3M_t + 5\epsilon M_t]\cap \Z$,
\[
 \big|\int_{t_1}^{t_2}   {  V_3(s )}  \big(P_k{E} (s,X(s))\big)_3 ds \big|. 
\]

 Based on the size of $\slashed \xi$, we decompose $ \big({E}_{k;j_1,j_2}(s,X(s))\big)_3$ further as follows,   
\begin{multline}\label{june10eqn43}
 \big( P_k{E} (s,X(s)) \big)_3 = \sum_{\tilde{k}\in [0, k+2]\cap \Z}   \tilde{E}_{k,\tilde{k}}(s,x) ,\quad  \tilde{E}_{k,\tilde{k}}(s,x) = \sum_{j_2\in[0, (1+\epsilon)M_t]\cap \Z, j_1\in[0, j_2+2]\cap \Z}  \tilde{E}_{k,\tilde{k}}^{j_1,j_2}(s,x),\\ 
  \tilde{E}_{k,\tilde{k}}^{j_1,j_2}(s,x)= \int_{\R^3} \int_{\R^3}  {K}_{k,\tilde{k}}(y) f(s,x-y, v) \varphi_{j_1}(\slashed v )  \varphi_{j_2}(v) d y d v, \quad {K}_{k,\tilde{k}}(y)= \int_{\R^3} e^{i x\cdot \xi} \varphi_{\tilde{k}}(\slashed \xi) \varphi_{k}(\xi) \frac{i\xi_3}{|\xi|^2}  d \xi. 
\end{multline}
 By doing integration by parts in $\xi$, we have
\be\label{june10eqn42}
|  {K}_{k,\tilde{k}}(y) |\lesssim 2^{2\tilde{k} } (1+2^{\tilde{k}}|\slashed y|)^{-N_0^3}(1+2^{  {k}}|  y_3|)^{-N_0^3}  .
 \ee

 Based on the possible size of $\tilde{k}$, we split into two cases as follows. 

$\bullet$\qquad If $\tilde{k}\leq k - 30\epsilon M_t.$

From the estimate of kernel in (\ref{june10eqn42}), the cylindrical symmetry of solution, and the estimate (\ref{jan12eqn36}) in Lemma \ref{spacetimeest}, we have 
\[
\int_{t_1}^{t_2}\|\tilde{E}_{k,\tilde{k}}^{j_1,j_2}(s,\cdot) \|_{L^\infty_x}  d s \lesssim  2^{-10M_t}+ \int_{t_1}^{t_2} \int_{\R^3} \int_{|y|\leq 2^{-\tilde{k}+\epsilon M_t}} \min\{2^{2\tilde{k}}, \frac{2^{\tilde{k}+\epsilon M_t/2}}{|\slashed x - \slashed y|}\} f(s,x-y, v) \varphi_{j_1}(\slashed v ) \varphi_{j_2}(v) d y d v d  s 
\]
\be\label{2022jan18eqn86}
\lesssim  2^{-10M_t} +2^{\tilde{k}-2j_1+j_2+6\epsilon M_t} .
\ee
Alternatively, if we use the volume of support of $v$, the following estimate holds from the estimate of kernel in (\ref{june10eqn42}), 
\be\label{june10eqn54} 
\int_{s_1}^{s_2}\|\tilde{E}_{k,\tilde{k}}^{j_1,j_2}(s,\cdot) \|_{L^\infty_x} ds \lesssim 2^{\epsilon M_t} \min\{ 2^{-k+2j_1+j_2}, 2^{2\tilde{k}-j_2}, 2^{2\tilde{k}-nj_2+(n-1)M_t}\}. 
\ee
After using  the estimates (\ref{2022jan18eqn86}) and (\ref{june10eqn54}) for the case $j_2\in[0,(1+\epsilon)M_t]\cap \Z$ and using the same argument used in obtaining the estimate (\ref{2022jan18eqn30}) for the case $[1+\epsilon)M_t, \infty)\cap \Z$,  we have 
\[
\sum_{j_2\in \Z_+, j_1\in[0, j_2+2]\cap \Z} \int_{s_1}^{s_2}\|\tilde{E}_{k,\tilde{k}}^{j_1,j_2}(s,\cdot) \|_{L^\infty_x} ds \lesssim \sum_{j_2\in [0, (1+\epsilon)M_t] \cap \Z_+, j_1\in[0, j_2+2]\cap \Z}    \min\{  2^{-10M_t} +2^{\tilde{k}-2j_1+j_2+6\epsilon M_t},  2^{\epsilon M_t} 2^{-k+2j_1+j_2} \}  
\]
\be\label{2022jan29eqn1}
+1 \lesssim 1+2^{3.5\epsilon M_t +(\tilde{k}-k)/2+(1+\epsilon)M_t}\lesssim 2^{(1-10\epsilon) M_t}. 
\ee
Therefore, recall  (\ref{june10eqn43}), from the above estimate, we have 
\be\label{2022jan20eqn72}
\int_{t_1}^{t_2} \big|  {  V_3(s )}      \tilde{E}_{k,\tilde{k}}^{}(s,X(s))  \big| ds\lesssim 2^{(2\gamma-5\epsilon)M_t}. 
\ee

$\bullet$\qquad If $\tilde{k}\in [ k - 30\epsilon M_t, k+2]\cap \Z$. 

\textbf{Only for this case, we exploit the hyperbolic nature of the sRVM system (\ref{sRVM}).} Recall (\ref{sRVM}) and (\ref{june10eqn43}). Note that, the following wave equation  holds for $ \tilde{E}_{k,\tilde{k}} (s,x),$
\[
\p_s^2   \tilde{E}_{k,\tilde{k}} - \Delta  \tilde{E}_{k,\tilde{k}}  =  - 4\pi \displaystyle{\int_{\R^3}\int_{\R^3} {K}_{k,\tilde{k}}(y) \p_s f(s, x-y, v) \hat{v}_3 d v dy } - 4\pi \int_{\R^3}\int_{\R^3}  {K}_{k,\tilde{k}}(y) \p_{x_3}f(s,x-y,v) d v dy
\]
\be\label{june10eqn66}
= \mathcal{N}_{k,\tilde{k}}^1(s, x) + \mathcal{N}_{k,\tilde{k}}^2(s, x),
\ee
where
\be\label{june10eqn67}
\mathcal{N}_{k,\tilde{k}}^1(s, x)= 4\pi\big(\int_{\R^3}\int_{\R^3} {K}_{k,\tilde{k}}(y)  \hat{v}_3  \hat{v}\cdot\nabla_x f(s,x-y,v)d vdy -  \int_{\R^3} \int_{\R^3}{K}_{k,\tilde{k}}(y) \p_{x_3}f(s,x-y,v)d v dy  \big) ,
\ee
\be\label{june10eqn31}
\mathcal{N}_{k,\tilde{k}}^2(s, x)= 4\pi \int_{\R^3} \int_{\R^3} {K}_{k,\tilde{k}}(y) E(s,x-y)\cdot \nabla_v \hat{v}_3 f(s,x-y,v) d v dy . 
\ee
Let 
\be\label{2022jan20eqn51}
\mathcal{E}_{k,\tilde{k}}(s):=(\p_s + i\d) \tilde{E}_{k,\tilde{k}}(s), \quad \tilde{E}_{k,\tilde{k}}(s)=  (2i\d )^{-1}(\mathcal{E}_{k,\tilde{k}}(s)- \overline{ \mathcal{E}_{k,\tilde{k}}(s)} ).  
\ee
From the equality (\ref{june10eqn66}),  we have 
\be\label{2022jan20eqn41}
(\p_s - i\d) \mathcal{E}_{k,\tilde{k}}(t)= \mathcal{N}_{k,\tilde{k}}^1(s) + \mathcal{N}_{k,\tilde{k}}^2(s),\quad \Longrightarrow 
 \mathcal{E}_{k,\tilde{k}}(s)= e^{i s\d}   \mathcal{E}_{k,\tilde{k}}(0) + \int_{0}^s e^{i(s-\tau)\d}   \big(\mathcal{N}_{k,\tilde{k}}^1(\tau) + \mathcal{N}_{k,\tilde{k}}^2(\tau) \big) d \tau.
\ee
After localizing the size of  $v$, we have 
\begin{multline}\label{june11eqn5}
\forall i \in \{1, 2\}, \quad \int_{0}^s e^{i(s-\tau)\d} \d^{-1}\mathcal{N}_{k,\tilde{k}}^i(\tau)  d\tau  = \sum_{j \in \Z_+ } T^{ j ;i}_{k, \tilde{k}}(s,x),\\
T^{j ;1}_{k, \tilde{k}}(s,x)= \int_0^s \int_{\R^3} \int_{\R^3}\int_{\R^3}  e^{i x\cdot \xi + i(s-\tau)|\xi |} \frac{ \nabla_v \hat{v}_3 }{|\xi|} \cdot \hat{E}(\tau, \eta) \varphi_{\tilde{k}}(\slashed \xi) \varphi_k(\xi)  \varphi_{j}(v)\hat{f}(\tau, \xi-\eta, v) d\eta  d \xi d v  d\tau,\\ 
T^{j ;2}_{k, \tilde{k}}(s,x)= \int_0^s \int_{\R^3} \int_{\R^3} e^{i x\cdot \xi + i(s-\tau)|\xi |} \frac{i(\hat{v}_3 \hat{v}\cdot \xi - \xi_3) }{|\xi|} \varphi_{\tilde{k}}(\slashed \xi) \varphi_k(\xi) \varphi_{j}(v)\hat{f}(\tau, \xi, v)  d \xi d v  d \tau.
\end{multline}
From the above equality,  (\ref{2022jan20eqn51}) and the fact that $V_3(s)\in \R$, we have
\be\label{2022jan20eqn65}
 \big|  \int_{t_1}^{t_2}{ V_3(s )}       \tilde{E}_{k,\tilde{k}}^{}(s,X(s))   ds \big| \lesssim 1 + \sum_{i=1,2}\sum_{j\in \Z} \big|\int_{t_1}^{t_2}{ V_3(s )}  T^{ j ;i}_{k, \tilde{k}}(s,X(s)) d s  \big|.
\ee

We first rule out the case $j \geq (1+\epsilon)M_t$. From the volume of support of $\xi$ and the rough estimate of the electric field (\ref{june2eqn71}) in Lemma \ref{roughgeneral},  we have
 \[
\sum_{j\geq (1+ \epsilon)M_t}   \| T^{j ; 1}_{k, \tilde{k}}(s,\cdot)\|_{L^\infty_x}+ \| T^{j ; 2}_{k, \tilde{k}}(s,\cdot)\|_{L^\infty_x} 
 \lesssim \sup_{\tau\in [0, t]} \sum_{j\geq (1+\epsilon)M_t}     \|  {f}(\tau, x, v) \varphi_{j}(v)\|_{L^1_xL^1_v} \big(2^{3k+\epsilon M_t} + 2^{2k-j+\epsilon M_t}  \| E(\tau, \cdot)\|_{L^\infty_x}\big)
\]
\be\label{june11eqn11}
\lesssim  \sum_{j \geq (1+\epsilon)M_t} 2^{9M_t +2\epsilon M_t} 2^{-nj+(n-1)M_t}\lesssim 1. 
\ee

Now, we focus on the case $j\in [0, (1+\epsilon)M_t]\cap \Z$.  Recall (\ref{june11eqn5}). Note that, 
\[
e^{i X(s)\cdot \xi + i(s-\tau)|\xi |}= \frac{1}{i |\xi| + i \hat{V}(s)\cdot \xi } \p_s e^{i X(s)\cdot \xi + i(s-\tau)|\xi |}.
\]
Hence, after doing integration by parts in $s$, we have
\be\label{june14eqn93}
\int_{t_1}^{t_2}   {  V_3(s )}  T^{j ;1}_{k, \tilde{k}} (s,X(s ))ds =  \mathfrak{H}^{j ;1}_{k, \tilde{k}}(t_1,t_2)+ \mathfrak{H}^{j ;2}_{k, \tilde{k}}(t_1,t_2)   + \mathfrak{High}^{j  }_{k, \tilde{k}}(t_1,t_2)),
\ee
 \be\label{june14eqn90}
\int_{t_1}^{t_2}    {  V_3(s )}  T^{j ;2}_{k, \tilde{k}} (s,X(s ))ds =  H^{j ;1}_{k, \tilde{k}}(t_1,t_2)+ H^{j ;2}_{k, \tilde{k}}(t_1,t_2)   + High^{j  }_{k, \tilde{k}}(t_1,t_2),
\ee
where
\be\label{june14eqn100}
\mathfrak{H}^{j ;1}_{k, \tilde{k}}(t_1,t_2)= \int_{t_1}^{t_2}     \int_{(\R^3)^3}  e^{i X(s)\cdot \xi    }  \frac{ i  {  V_3(s )}   \nabla_v \hat{v}_3 \cdot \hat{E}(s, \eta)  }{|\xi|( |\xi| +  \hat{V}(s )\cdot \xi)}  \hat{f}(s, \xi-\eta, v)  \varphi_{\tilde{k}}(\slashed \xi) \varphi_k(\xi) \varphi_{j}(v) d\eta  d \xi d v  d s,
\ee
\be\label{june14eqn101}
\mathfrak{H}^{j ;2}_{k, \tilde{k}}(t_1,t_2)= \sum_{a=1,2}  \int_0^{t_a}   \int_{(\R^3)^3}   e^{i X(t_a)\cdot \xi + i(t_a-\tau)|\xi |} \frac{ i (-1)^{a-1}  {  V_3(t_a )}  \nabla_v \hat{v}_3 }{|\xi|( |\xi| +  \hat{V}(t_a )\cdot \xi)} \cdot \hat{E}(\tau, \eta) \hat{f}(\tau, \xi-\eta, v)  \varphi_{\tilde{k}}(\slashed \xi)\varphi_k(\xi) \varphi_{j}(v)  d\eta  d \xi d v  d \tau,
\ee
\be\label{june14eqn102}
\mathfrak{High}^{j  }_{k, \tilde{k}}(t_1,t_2)=  \int_{t_1}^{t_2}   \int_0^s      \int_{(\R^3)^3}  e^{i X(s)\cdot \xi + i(s-\tau)|\xi |} \p_s\big[ \frac{ i  {  V_3(s )}   \nabla_v \hat{v}_3 }{|\xi|( |\xi| +  \hat{V}(s )\cdot \xi)}\big] \cdot \hat{E}(\tau, \eta)  \hat{f}(\tau, \xi-\eta, v)  \varphi_{\tilde{k}}(\slashed \xi)   \varphi_k(\xi) \varphi_{j}(v) d\eta  d \xi d v  d \tau d s . 
 \ee
\be\label{june10eqn71}
H^{j ;1}_{k, \tilde{k}}(t_1,t_2)= \int_{t_1}^{t_2}     \int_{\R^3} \int_{\R^3} e^{i  X(s) \cdot \xi } \frac{V_3(s) (\hat{v}_3 \hat{v}\cdot \xi - \xi_3) }{|\xi|( |\xi| +  \hat{V}(s )\cdot \xi)}     \varphi_{\tilde{k}}(\slashed \xi)   \varphi_k(\xi) \varphi_{j}(v) \hat{f}(s, \xi, v)  d \xi d v   d s, 
\ee
\be\label{june11eqn33} 
H^{j ;2}_{k, \tilde{k}}(t_1,t_2)= \sum_{a=1,2}  \int_0^{t_a} \int_{\R^3} \int_{\R^3} e^{i  X(t_a)\cdot \xi + i(t_a-\tau)|\xi |} \frac{ (\hat{v}_3 \hat{v}\cdot \xi - \xi_3) }{|\xi|}  \frac{(-1)^{a-1}  V_3(t_a )}{  |\xi| +  \hat{V}(t_a)\cdot \xi }\varphi_{\tilde{k}}(\slashed \xi)  \varphi_k(\xi) \varphi_{j}(v)\hat{f}(\tau, \xi, v)  d \xi d v  d \tau,
\ee
\be\label{june10eqn72}
High^{j  }_{k, \tilde{k}}(t_1,t_2)= \int_{t_1}^{t_2}      \int_0^s \int_{\R^3} \int_{\R^3} e^{i X(s) \cdot \xi + i(s-\tau)|\xi |}  \p_s \big[\frac{V_3(s) (\hat{v}_3 \hat{v}\cdot \xi - \xi_3) }{|\xi|( |\xi| +  \hat{V}(s )\cdot \xi)}\big]   \varphi_k(\xi) \varphi_{j}(v)  \varphi_{\tilde{k}}(\slashed \xi)  \hat{f}(\tau, \xi, v)  d \xi d v  d \tau d s. 
\ee 

  \noindent $\oplus$\qquad The estimate of $\mathfrak{H}^{j ;1}_{k, \tilde{k}}(s_1,s_2)$. 

Recall (\ref{june14eqn100}). In terms of kernel, we have
\begin{multline}\label{2022jan18eqn91}
\mathfrak{H}^{j ;1}_{k, \tilde{k}}(t_1,t_2) =\int_{t_1}^{t_2}    \int_{\R^3} \int_{\R^3}  {  V_3(s )} \mathfrak{K}^{j ;0}_{k, \tilde{k}}(y, v, V(s))\cdot  E(s, X(s)-y) f(s, X(s)-y, v)   dy d v d s,\\ 
\mathfrak{K}^{j ;0}_{k, \tilde{k}}(y, v, V(s)):= \int_{\R^3} e^{i y \cdot \xi}  \frac{ i     \nabla_v \hat{v}_3   }{|\xi|( |\xi| +  \hat{V}(s )\cdot \xi)}   \varphi_{\tilde{k}}(\slashed \xi) \varphi_k(\xi) \varphi_{j}(v) d \xi.
\end{multline}
Note that
\be\label{2022jan20eqn1}
\forall v\in supp(\varphi_{\tilde{j}}(\slashed v ) \varphi_{j}(v) ), \qquad   |\nabla_v \hat{v}_3|= |(\frac{-v_3 v_1}{\langle v \rangle^3}, \frac{-v_3 v_2}{\langle v \rangle^3},  \frac{v_1^2+v_2^2}{\langle v \rangle^3} )|\lesssim 2^{\tilde{j}-2j}. 
\ee
After doing integration by parts in $\xi$ many times, from the above estimate,  the following estimate holds for the kernel $\mathfrak{K}^{j ;0}_{k, \tilde{k}}(y, v, V(s))$, 
\[
| \mathfrak{K}^{j ;0}_{k, \tilde{k}}(y, v, V(s))|\varphi_{\tilde{j}}(\slashed v )  \lesssim 2^{k+ \tilde{j}-2j +90\epsilon M_t} (1+2^{k-30\epsilon M_t}|y|)^{-N_0^3}\varphi_{\tilde{j}}(\slashed v) \varphi_{  {j}}(  v ).  
\]

From the above estimate of kernel and the estimate (\ref{june14eqn21}) in Lemma \ref{roughgeneral}, we have
\be\label{june14eqn140}
|\mathfrak{H}^{j ;1}_{k, \tilde{k}}(t_1,t_2)|\lesssim \sum_{\tilde{j}\in [0, j+2]\cap \Z}  \int_{t_1}^{t_2} 2^{k+ \tilde{j}-2j +  \gamma M_t+100\epsilon M_t}\min\{2^{-3k+2 \tilde{j} +j}, 2^{-j} \} \| E(s,\cdot)\|_{L^\infty_x}  d s \lesssim 2^{-k/2+ 2M_t + 200\epsilon M_t}\lesssim 2^{7M_t/4 }. 
\ee

\noindent $\oplus$\qquad The estimate of $\mathfrak{H}^{j ;2}_{k, \tilde{k}}(t_1,t_2)$. 

Recall (\ref{june14eqn101}). From the Kirchhoff's formulas in Lemma \ref{Kirchhoff}, in terms of kernel, we have
\[
\mathfrak{H}^{j;2}_{k, \tilde{k}}(t_1,t_2)= \sum_{\tilde{j}\in[0, j+2]\cap \Z} \mathfrak{H}^{\tilde{j}, j;2}_{k, \tilde{k}}(t_1,t_2) , \quad \mathfrak{H}^{\tilde{j}, j;2}_{k, \tilde{k}}(t_1,t_2):= \sum_{a=1,2}  \int_0^{t_a}  \int_{\R^3} \int_{\R^3} \int_{\mathbb{S}^2}  (-1)^{a-1}    V_3(t_a ) \big((t_a-\tau)\mathfrak{K}^{j ;1}_{k, \tilde{k}}(y, v, V(t_a),\omega)
\]
\be\label{june14eqn121}
 + \mathfrak{K}^{j ;2}_{k, \tilde{k}}(y, v, V(t_a), \omega)  \big)E(\tau, X(t_a)-y+(t_a-\tau)\omega) f(\tau, X(t_a)-y+(t_a-\tau)\omega, v)\varphi_{\tilde{j}}(\slashed v)  d\omega dy d v d\tau, 
\ee
where the kernels $\mathfrak{K}^{j_1,j_2;i}_{k, \tilde{k}}(y, v, V(s)), i\in \{1,2\}$,  are defined as follows, 
\be\label{june14eqn110}
\mathfrak{K}^{j ;1}_{k, \tilde{k}}(y, v, V(s), \omega)=  \int_{\R^3} e^{i y \cdot \xi}  \frac{ (\xi\cdot \omega + |\xi|)    \nabla_v \hat{v}_3   }{|\xi|( |\xi| +  \hat{V}(s )\cdot \xi)}   \varphi_{\tilde{k}}(\slashed \xi) \varphi_k(\xi)   \varphi_{j}(v) d \xi,
\ee
\be\label{june14eqn111}
\mathfrak{K}^{j ;2}_{k, \tilde{k}}(y, v, V(s), \omega)=  \int_{\R^3} e^{i y \cdot \xi}  \frac{     \nabla_v \hat{v}_3   }{|\xi|( |\xi| +  \hat{V}(s )\cdot \xi)}   \varphi_{\tilde{k}}(\slashed \xi) \varphi_k(\xi)  \varphi_{j}(v) d \xi. 
\ee
By doing integration by parts in $\xi$ many times, from the estimate (\ref{2022jan20eqn1}),  we have
\[
\big(| \mathfrak{K}^{j ;1}_{k, \tilde{k}}(y, v, V(s), \omega)|+ 2^k | \mathfrak{K}^{j ;2}_{k, \tilde{k}}(y, v, V(s), \omega)|\big)\varphi_{\tilde{j}}(\slashed v)  \varphi_{j}(v)  \lesssim 2^{2k+\tilde{j}-2j +90\epsilon M_t} (1+2^{k-30\epsilon M_t}|y|)^{-N_0^3} \varphi_{\tilde{j}}(\slashed v)  \varphi_{j}(v). 
\]

From the above estimates of kernels, the following estimate holds after using the volume of support of $v$ and the estimate  (\ref{june14eqn21}) in Lemma \ref{roughgeneral},  
\[
\big|\mathfrak{H}^{\tilde{j}, j;2}_{k, \tilde{k}}(t_1,t_2)\big|\lesssim  \sum_{a=1,2}  \int_0^{t_a}  \int_{\R^3} \int_{\R^3} \int_{0}^{2\pi} \int_0^{\pi}  2^{k+ \tilde{j}-2j +\gamma M_t+90\epsilon M_t}  ((t_a-\tau) 2^k + 1) (1+2^{k-30\epsilon M_t}|y|)^{-N_0^3}\| E(\tau, \cdot)\|_{L^\infty_x}
\]
\be\label{june14eqn112}
\times f(\tau, X(t_a)-y+(t_a-\tau)(\sin\theta\cos\phi, \sin\theta\sin \phi, \cos\theta) , v)  \varphi_{\tilde{j}}(\slashed v)  \varphi_{j}(v) \sin \theta d\theta d\phi  dy d v d\tau
\ee
\be\label{june14eqn113}
\lesssim \sum_{a=1,2}  2^{-k+2\tilde{j}+j}  2^{\tilde{j}-2j  +\gamma M_t+40\epsilon M_t} \int_0^{s_a}  \| E(\tau, \cdot)\|_{L^\infty_x}  d\tau \lesssim 2^{-k+ 3\tilde{j}-j   +2\gamma M_t +100\epsilon M_t}. 
\ee

Alternatively, if we use the cylindrical symmetry of the distribution function, the rough estimate (\ref{june2eqn71}) in Lemma \ref{roughgeneral} for the electric field,     and the estimate (\ref{jan12eqn36}) in Lemma \ref{spacetimeest} and change coordinates $\theta\rightarrow X_3(t_a)-y_3+(t_a-\tau)\cos\theta$,  the following estimate holds, 
\[
 \big|\mathfrak{H}^{\tilde{j}, j;2}_{k, \tilde{k}}(t_1,t_2)\big| \lesssim  \sum_{a=1,2}  \int_0^{t_a}  \int_{\R^3} \int_{\R^3} \int_{0}^{2\pi} \int_0^{\pi}  2^{\tilde{j}-2j +\gamma M_t+40\epsilon M_t}  ((t_a-\tau) 2^k + 1)  \varphi_{\tilde{j}}(\slashed v)  \varphi_{j}(v)  (1+2^{k-20\epsilon M_t}|y_3|)^{-N_0^3}
\] 
\[
\times \frac{ 2^{4  {\alpha}_{\tau} M_\tau /3 + M_t/3+5\epsilon M_t}}{|(z_1,z_2)|^{1-2\epsilon}} f(\tau, z_1,z_2,  X_3(t_a)-y_3+(t_a-\tau)\cos\theta) \sin \theta d \theta  d\phi d z_1 dz_2 dy d v d \tau  + 2^{-10M_t}
\]
\be\label{june14eqn114}
\lesssim 2^{4\alpha^{\star} M_t/3+ M_t/3+\gamma M_t+100\epsilon M_t}2^{\tilde{j}-2j}2^{-2\tilde{j}+j}+ 2^{-10M_t} .
\ee
After combining the estimates (\ref{june14eqn113}) and (\ref{june14eqn114}), we have
\[
 \big|\mathfrak{H}^{  j;2}_{k, \tilde{k}}(t_1,t_2)\big| \lesssim  \sum_{\tilde{j}\in[0, j+2]\cap \Z}   2^{\gamma M_t+ 100\epsilon M_t}\min\{ 2^{-k+2\tilde{j} + \gamma M_t }, 2^{4\alpha^{\star} M_t/3+  M_t/3}2^{-\tilde{j}-j} +2^{-10M_t}\}
\]
\be\label{june14eqn133}
\lesssim 2^{-k/2+5 \gamma M_t/3+2\alpha^{\star} M_t/3+200\epsilon M_t} \lesssim 2^{1.9\gamma M_t}. 
\ee

 \noindent $\oplus$\qquad The estimate of $ \mathfrak{High}^{ j  }_{k, \tilde{k}}(s_1,s_2)$. 

Recall (\ref{june14eqn102}). As a result of direct computations, we have
\be\label{2022jan29eqn5}
 \p_s\big[ \frac{ i  {  V_3(s )}   }{|\xi|( |\xi| +  \hat{V}(s )\cdot \xi)}\big] = \langle V(s)\rangle^{-1} E(s, X(s)) \cdot \mathfrak{M}(\xi, v, V(s)),
\ee
\[
\mathfrak{M}(\xi, v, V(s))= \frac{  \mathbf{e}_3 \langle V(s)\rangle   }{|\xi|( |\xi| +  \hat{V}(s )\cdot \xi)} - \frac{V_3(s)    }{|\xi|( |\xi| +  \hat{V}(s )\cdot \xi)^2} \big(\xi - \hat{V}(s) \hat{V}(s)\cdot \xi \big).
\]
Hence 
\[
\mathfrak{High}^{j   }_{k, \tilde{k}}(t_1,t_2)= \int_{t_1}^{t_2} \langle V(s)\rangle^{-1} E(s, X(s))\cdot \widetilde{\mathfrak{High}}^{j   }_{k, \tilde{k}}(s) d s, 
\]
where
\be\label{june14eqn132}
\widetilde{\mathfrak{High}}^{j  }_{k, \tilde{k}}(s) :=  \int_0^s    \int_{\R^3} \int_{\R^3}\int_{\R^3}  e^{i X(s)\cdot \xi + i(s-\tau)|\xi |} \nabla_v \hat{v}_3 \mathfrak{M}(\xi, v, V(s))\cdot \hat{E}(\tau, \eta) \hat{f}(\tau, \xi-\eta, v)  \varphi_{\tilde{k}}(\slashed \xi) \varphi_k(\xi)  \varphi_{j}(v) d\eta  d \xi d v   d\tau . 
\ee
 Note that, recall (\ref{june14eqn101}), the difference between $\widetilde{\mathfrak{High}}^{j  }_{k, \tilde{k}}(s) $ and $ \mathfrak{H}^{j ;2}_{k, \tilde{k}}(t_1,t_2)$ lies only in the symbol and the characteristic time, which don't play much role in the proof of  the obtained estimate (\ref{june14eqn133}). With minor modifications, we have
\[
 \forall s\in[t_1,t_2], \quad  | \widetilde{\mathfrak{High}}^{j   }_{k, \tilde{k}}(s)|\lesssim  2^{1.9\gamma M_t}. 
\]
From the above estimate and the estimate (\ref{june14eqn21}) in Lemma \ref{roughgeneral}, we have
\be\label{2022jan25eqn21}
\big|\mathfrak{High}^{j_1,j_2  }_{k, \tilde{k}}(t_1,t_2)\big|\lesssim  \int_{t_1}^{t_2}2^{0.9 \gamma M_t  +5\epsilon M_t} \| E(s, \cdot)\|_{L^\infty_x} d s\lesssim  2^{1.9 \gamma M_t +40\epsilon M_t}. 
\ee
 
\noindent $\oplus$\qquad The estimate of $H^{j ;1}_{k, \tilde{k}}(t_1,t_2).$

Recall (\ref{june10eqn71}). In terms of kernel, we have 
\[
H^{j ;1}_{k, \tilde{k}}(t_1,t_2)= \int_{t_1}^{t_2}  
\int_{\R^3} \int_{\R^3}   \mathcal{K}_{\tilde{k}, k}(y, v, V(s)) f(s, X(s)-y, v)  \varphi_{j}(v) d y d v d s,  
\]
where
\[
\mathcal{K}_{\tilde{k}, k}(y, v,  V(s)) := \int_{\R^3} e^{i y\cdot \xi }\frac{V_3(s) (\hat{v}_3 \hat{v}\cdot \xi - \xi_3) }{|\xi|( |\xi| +  \hat{V}(s )\cdot \xi)}     \varphi_{\tilde{k}}(\slashed \xi) \varphi_k(\xi) d \xi. 
\]
Due to the facts that $|\slashed V(s)|\leq  2^{\alpha^{\star}M_t}, |V(s)|\gtrsim  2^{(1-3\epsilon) M_t}$, and $|\slashed \xi|/|\xi|\gtrsim 2^{-30\epsilon M_t}$,  for any $(v,\xi)\in supp(    \varphi_{\tilde{k}}(\slashed \xi) \varphi_k(\xi) \varphi_{\tilde{j}}(\slashed v ) \varphi_{j}(v))$,   we have
\be\label{june11eqn31}
 | \tilde{\xi}\times \tilde{V}(s)|\gtrsim  2^{-30\epsilon M_t}, \quad |\hat{v}_3 \hat{v}\cdot \xi - \xi_3|\leq |\xi_3|(1-\hat{v}_3^2) + |\hat{v}_1\xi_1| + |\hat{v}_2\xi_2| \lesssim 2^{\tilde{j}-j+k }. 
\ee
Therefore, after doing integration by parts in $\xi$ along $V(s)$ and directions perpendicular to $V(s)$, the following estimate holds for the kernel, 
\be\label{june10eqn81}
\big| \mathcal{K}_{\tilde{k}, k}(y, v,  V(s))\big| \varphi_{\tilde{j}}(\slashed v ) \varphi_{j}(v)   \lesssim 2^{2{k}+200\epsilon M_t+\tilde{j}-j} (1+2^{{k}-30\epsilon M_t} |y|)^{-N_0^3}  \varphi_{\tilde{j}}(\slashed v ) \varphi_{j}(v).
\ee 
From the above estimate of the kernel, the volume of support of  $v$ and the estimate (\ref{outsidemajority}) if $|\slashed v|\geq 2^{(\alpha^{\star}+\epsilon)M_t}$,  after rerunning the argument used in obtaining the estimate (\ref{june10eqn54}),  we have 
 \be\label{june11eqn32}
\big| H^{j ;1}_{k, \tilde{k}}(t_1,t_2)\big| \lesssim \sum_{\tilde{j}\in [0, j+2]\cap \Z}   2^{\gamma M_t + 300\epsilon M_t+ \tilde{j}-j} \min\{2^{-\tilde{k}} 2^{2(\alpha^{\star} +\epsilon)M_t + j},  2^{\tilde{k}-2\tilde{j}+j} \} \lesssim 2^{1.8\gamma M_t}. 
\ee

For better presentation, we defer the estimate of remaining terms  $H^{j ;2}_{k, \tilde{k}}(t_1,t_2)$ and $ High^{j  }_{k, \tilde{k}}(t_1,t_2)$ to Lemma \ref{keylemma2} and Lemma \ref{keylemma3} respectively.

To sum up, recall (\ref{2022jan20eqn65}),  (\ref{june14eqn93}) and (\ref{june14eqn90}), our desired estimate (\ref{2022jan18eqn51}) holds  from the obtained estimates (\ref{2022jan20eqn71}), (\ref{2022jan20eqn72}), (\ref{2022jan20eqn65}),    (\ref{june11eqn11}),  (\ref{june14eqn140}), (\ref{june14eqn133}), (\ref{2022jan20eqn31}), (\ref{june11eqn32}), the estimate in Lemma \ref{june14eqn65}, and the estimate   (\ref{june14eqn72}) in Lemma \ref{keylemma3}.

 \qed 

 \vo 
 
 Now, we give the deferred estimate of $H^{j ;2}_{k, \tilde{k}}(t_1,t_2)$ in the following Lemma.

  \begin{lemma}\label{keylemma2}
 Let $j\in [0, (1+\epsilon) M_t]\cap \Z$,  $k\in  [4M_t/5-20\epsilon M_t,3M_t+5\epsilon M_t]\cap \Z$ and $\tilde{k}\in[0, k+2]\cap \Z$, s.t., $\tilde{k}\geq k-30\epsilon M_t$. Under the assumption of Lemma \textup{\ref{bootstraplemma2}}, we have
 \be\label{june14eqn65}
 \big| H^{j ;2}_{k, \tilde{k}}(t_1,t_2)\big|  \lesssim 2^{1.9 \gamma M_t + 10\epsilon M_t}. 
\ee
 \end{lemma}

 \begin{proof}
Recall (\ref{june11eqn33}). Note that $t \leq 2^{\epsilon M_t-10}$. We localize further the size of   $t_a-\tau$ and the  angle between $\xi/|\xi|$ and $v/|v|$ as follows, 
\be\label{june14eqn63}
H^{j ;2}_{k, \tilde{k}}(t_1,t_2)=  \sum_{m\in [-10M_t, \epsilon M_t]\cap \Z }\sum_{l\in [\bar{l}, 2]\cap \Z} H^{j ;2}_{k, \tilde{k};m, l}(t_1,t_2),
\ee
where $\bar{l}:= -j + 5M_t/18$ and $H^{j ;2}_{k, \tilde{k};l}(s_1,s_2)$ is defined as follows, 
\[
H^{j ;2}_{k, \tilde{k};m,l}(t_1,t_2)=  \sum_{a=1,2} (-1)^{a-1}  V_3(t_a ) \int_0^{t_a} \int_{\R^3} \int_{\R^3} e^{i  X(t_a)\cdot \xi + i(t_a-\tau)|\xi |} \frac{ (\hat{v}_3 \hat{v}\cdot \xi - \xi_3) }{|\xi|(|\xi| +  \hat{V}(t_a)\cdot \xi)}  \varphi_{l;\bar{l}}(\frac{\xi \times v }{|\xi||v|})\varphi_{m;-10M_t}(t_a-\tau)  \] 
\be\label{june13eqn21}
\times \varphi_{\tilde{k}}(\slashed \xi) \varphi_k(\xi) \varphi_{j}(v)\hat{f}(\tau, \xi, v)  d \xi d v  d \tau.
\ee

From the Kirchhoff's formula in Lemma \ref{Kirchhoff}, after representing  $H^{j_1,j_2;2}_{k, \tilde{k};m,l}(t_1,t_2)$  in terms of kernel, the following estimate holds from the volume of support of $v$ and $\tau$, 
\[
\big|H^{j ;2}_{k, \tilde{k};m,l}(t_1,t_2)\big|\lesssim 2^{m+200\epsilon M_t}2^{3j}.
\]
From the above estimate, we can rule out the case $m=-10M_t$ and focus on the case $m\in (-10M_t, \epsilon M_t]\cap \Z$, in which we have $t_a-\tau \sim 2^m.$

Based on the possible size of $l$, we separate into two cases as follows.  

\noindent $\oplus$\qquad If $l=\bar{l}=-j + 5M_t/18$.

Recall the equality  (\ref{march4eqn41})  in the proof of Kirchhoff's formula.  From the stationary point of view, we know that the angle between $\theta$ and $\xi$ are localized around the size of $ 2^{-(m+k)/2}$.  Therefore, after localizing the angle between $v$ and $\theta$,  the following equality holds from the Kirchhoff's formulas in Lemma \ref{Kirchhoff}, 
\begin{multline}\label{june14eqn59}
H^{j ;2}_{k, \tilde{k};m,l}(t_1,t_2)= \sum_{\ast\in \{\textup{ess}, \textup{err}\}} \mathcal{H}^{\ast; j ;1 }_{k, \tilde{k};m,l}(t_1,t_2) + \mathcal{H}^{\ast; j ;2 }_{k, \tilde{k};m,l}(t_1,t_2),\quad  \mathcal{H}^{\ast; j ;i }_{k, \tilde{k};l}(t_1,t_2)=  \sum_{a=1,2}  \int_0^{t_a} \int_{\R^3} \int_{\R^3} \int_{\mathbb{S}^2}  (-1)^{a-1} (t_a-\tau) \\ 
  \times V_3(t_a )\varphi_{m;-10M_t}(t_a-\tau) \mathcal{K}^{\ast;i}_{k, \tilde{k};l}(y, v, V(t_a), \omega)   f(\tau, X(t_a)-y+(t_a-\tau)\omega)  \varphi_{j}(v)  d \omega d y d v d \tau, \quad i\in\{1,2\},
\end{multline}
where the kernels $\mathcal{K}^i_{k, \tilde{k}}(y, v, V(t_a), \omega), i\in\{1,2\}, $ are defined as follows, 
\[
\mathcal{K}^{\ast;1}_{k, \tilde{k};l}(y, v, V(t_a),\omega):=  \int_{\R^3} e^{iy \cdot \xi }\frac{i(\xi\cdot \omega + |\xi|) (\hat{v}_3 \hat{v}\cdot \xi - \xi_3) }{4\pi |\xi| (  |\xi| +  \hat{V}(t_a)\cdot \xi )}    \varphi_{\tilde{k}}(\slashed \xi) \varphi_k(\xi) \varphi_{l;\bar{l}}(\frac{\xi \times v }{|\xi||v|}) \phi_{\ast}(  \omega, v ) d \xi,
\]
\[
\mathcal{K}^{\ast;2}_{k, \tilde{k};l}(y, v, V(t_a),\omega):=  \int_{\R^3} e^{iy \cdot \xi }\frac{    (\hat{v}_3 \hat{v}\cdot \xi - \xi_3) }{4\pi|\xi| (  |\xi| +  \hat{V}(t_a)\cdot \xi )}    \varphi_{\tilde{k}}(\slashed \xi) \varphi_k(\xi)  \varphi_{l;\bar{l}}(\frac{\xi \times v }{|\xi||v|}) \phi_{\ast}(  \omega, v ) d \xi,
\]
and   the cutoff functions $\phi_{\ast}(  \omega,   v), \ast\in \{\textup{ess}, \textup{err}\}$,  are defined as follows, 
\be\label{june13eqn1}
\phi_{\textup{ess}}( \omega,  v ):=\psi_{\leq \max\{l, -(m+k)_{+}/2\}+\epsilon M_t}(  \omega \times \tilde{v}   ), \quad \phi_{\textup{err}}(  \omega,  v ) =1-\phi_{\textup{ess}}(  \omega,  v ). 
\ee

We first rule out the error type terms. Due to the fact that $\angle(\xi, \pm \omega )\geq 2^{-(k+m)/2+\epsilon M_t/2}  $, we know that we gain at least $2^{-\epsilon M_t/2}$ each time we do integration by parts in $\omega$ once. Hence, after doing integration by parts in $\omega$ many times  and then do integration by parts in $\xi$, the following estimate holds for the  error type  kernels $\mathcal{K}^{err;i}_{k, \tilde{k};l}(\cdot,\cdot, \cdot),i\in\{1,2\}$,
\be\label{june13eqn11}
|\mathcal{K}^{err;1}_{k, \tilde{k};l}(y, v, V(t_a),\omega)| +2^k |\mathcal{K}^{err;2}_{k, \tilde{k};l}(y, v, V(t_a),\omega)|  \lesssim 2^{3k+2l-100M_t}  (1+2^{k+l-30\epsilon M_t}|y |)^{-N_0^3}. 
\ee
From the above  estimate of kernels, the following estimate holds from  the volume of support of $v$, 
\be\label{2022jan19eqn21}
 \big| H^{err;j ;1}_{k, \tilde{k};m,l}(t_1,t_2)\big|+  \big| H^{err;j ;2}_{k, \tilde{k};m,l}(t_1,t_2)\big|\lesssim 2^{-50 M_t+2j_1+j_2}\lesssim 1. 
\ee

Now, we focus on the essential type terms. Note that, for any $(\xi, v)\in supp(\varphi_k(\xi) \varphi_{l;\bar{l}}( \tilde{v}\times \tilde{\xi} )$, we have 
\[
\big|\hat{v}_3 \hat{v}\cdot \xi - \xi_3\big|\lesssim \min_{\mu\in\{+,-\}}|\xi|\big(|\hat{v}_3-\mu \xi_3/|\xi|| + |\hat{v}\cdot\xi/|\xi|-\mu|\big)\lesssim 2^{k+l}. 
\]
 After doing integration by parts in $\xi$ in $v$ direction and directions perpendicular to $v$,  from the above estimate and  the estimate  (\ref{june11eqn31}),  we have 
\[
\big(|\mathcal{K}^{ess;1}_{k, \tilde{k}}(y, v, V(t_a),\omega)| +2^k |\mathcal{K}^{ess;2 }_{k, \tilde{k}}(y, v, V(t_a),\omega)|\big) \varphi_{\tilde{j}}(\slashed v)\varphi_j(v) \lesssim 2^{3k+2l+120\epsilon M_t} \min\{2^{\tilde{j}-j}, 2^l\}\phi_{\textup{ess}}( \omega,  v )
\]
\be\label{june13eqn12}  
\times (1+2^{k-30\epsilon M_t}|y\cdot v|)^{-N_0^3}(1+2^{k+l-30\epsilon M_t}|y\times v|)^{-N_0^3} \varphi_{\tilde{j}}(\slashed v)\varphi_j(v).   
\ee 
From the above estimate of kernels and  the volume of support of $\omega$ and $v$, we have  
\[
\big| H^{ess;j ;1}_{k, \tilde{k};l}(t_1,t_2)\big|+ \big| H^{ess;j ;2}_{k, \tilde{k};l}(t_1,t_2)\big|\lesssim     2^{\gamma M_t + l+150\epsilon M_t} 2^{ m}  2^{2 \max\{l, -(m+k)_{+}/2\}} \min\{ 2^{3j}, 2^{3k+2l-j}\} \lesssim 2^{\gamma M_t + 160\epsilon M_t}\big( 2^{3l }  + 2^{-k+l }\big) 
\]
\be\label{june14eqn60}
 \times \min\{ 2^{3j}, 2^{3k+2l-j }\}\lesssim   2^{\gamma M_t +160\epsilon M_t}\big( 2^{3l+ 3j } + 2^{l+ 2j}2^{(2l-j)/3}\big) \lesssim  2^{M_t +160\epsilon M_t}( 2^{3l+3j}+ 2^{5l/3+5j/3})\lesssim 2^{1.9\gamma M_t}. 
\ee
 Recall the decomposition (\ref{june14eqn59}). After combining the above estimate with  the estimate  (\ref{2022jan19eqn21}), we have
\be\label{june14eqn70}
\big| H^{j;2}_{k, \tilde{k};m,l}(t_1,t_2)\big|\lesssim 2^{1.9\gamma M_t}. 
\ee

\noindent $\oplus$\qquad If $l > \bar{l}= -j + 5M_t/18$.

Recall (\ref{june13eqn21}). For this case, to exploit the oscillation in $\tau$, we do integration by parts in $\tau$ once. As a result, we have
\be\label{june14eqn56}
\big|H^{j ;2}_{k, \tilde{k};m,l}(t_1,t_2)\big|\leq \sum_{a=1,2}\big| \widetilde{End}^{j ;a}_{k, \tilde{k};m,l}(t_1,t_2) \big|  + \big|\tilde{\mathcal{H}}^{j ;a}_{k, \tilde{k};m,l}(t_1,t_2)\big|,
\ee
where
\[
 \widetilde{End}^{j ;a}_{k, \tilde{k};m,l}(t_1,t_2)   =     
\big| \int_0^{t_a}  \int_{\R^3} \int_{\R^3} e^{i  X(t_a)\cdot \xi + i(t_a-\tau)|\xi | }  \frac{  V_3(t_a ) (\hat{v}_3 \hat{v}\cdot \xi - \xi_3) \varphi_{l;\bar{l}}(\tilde{v}\times \tilde{\xi})   \varphi_{\tilde{k}}(\slashed \xi) \varphi_{k }(\xi) \varphi_{j }(v)  }{|\xi|(|\xi| +  \hat{V}(t_a)\cdot \xi)(|\xi|+\hat{v}\cdot \xi)}      
  \] 
  \[
\times  \p_\tau  \varphi_{m;-10M_t}(t_a-\tau) \hat{f}(\tau, \xi, v)d \xi d v d\tau\big| + \big| \int_{\R^3} \int_{\R^3} e^{i  X(t_a)\cdot \xi + i(t_a-\tau)|\xi |  } \frac{  V_3(t_a ) (\hat{v}_3 \hat{v}\cdot \xi - \xi_3) \varphi_{l;\bar{l}}(\tilde{v}\times \tilde{\xi})   \varphi_{\tilde{k}}(\slashed \xi)  \varphi_{k }(\xi) \varphi_{j }(v)  }{|\xi|(|\xi| +  \hat{V}(t_a)\cdot \xi)(|\xi|+\hat{v}\cdot \xi)} \]
\be 
\times  \varphi_{m;-10M_t}(t_a-\tau)\hat{f}(\tau, \xi, v)  d \xi d v \big|_{\tau=0}^{t_a}\big|,
\ee
 \[
\tilde{\mathcal{H}}^{j;a}_{k, \tilde{k};m,l}(t_1,t_2) =  \int_0^{t_a}  \int_{\R^3} \int_{\R^3} e^{i  X(t_a)\cdot \xi + i(t_a-\tau)|\xi |- i \tau \hat{v}\cdot \xi }  \frac{  V_3(t_a ) (\hat{v}_3 \hat{v}\cdot \xi - \xi_3) \varphi_{l;\bar{l}}(\tilde{v}\times \tilde{\xi})   \varphi_{\tilde{k}}(\slashed \xi)  \varphi_{k }(\xi) \varphi_{j }(v)   }{|\xi|(|\xi| +  \hat{V}(t_a)\cdot \xi)(|\xi|+\hat{v}\cdot \xi)}   \]
\be\label{2022jan19eqn31}
\times \varphi_{m;-10M_t}(t_a-\tau) \p_\tau \hat{g}(\tau, \xi, v)    d \xi d v d\tau.
\ee

From the Kirchhoff's formula in Lemma \ref{Kirchhoff}, we can represent  $ \widetilde{End}^{j_1,j_2;a}_{k, \tilde{k};m,l}(t_1,t_2)  $   in terms of kernel as follows, 
\[ 
 \big| \widetilde{End}^{j ;a}_{k, \tilde{k};m,l}(t_1,t_2)  \big|\lesssim 1+ \big|   \int_{\R^3} \int_{\R^3} \tilde{\mathcal{K}}^{ 0}_{k, \tilde{k};l}(y, v, V(t_a)) f(t_a, X(t_a)-y, v)  \varphi_{j }(v) dy d v\big| + \big|\int_0^{t_a}  \int_{\R^3} \int_{\R^3} \int_{\mathbb{S}^2}   \p_\tau  \varphi_{m;-10M_t}(t_a-\tau) 
\]
 \be\label{jan19eqn36}
 \times \big( \tilde{\mathcal{K}}^{ 0}_{k, \tilde{k};l}(y, v, V(t_a))+ (t_a-\tau) \tilde{\mathcal{K}}^{1}_{k, \tilde{k};l}(y, v, V(s_a), \omega )  \big)   \varphi_{j}(v) {f}(\tau, X(t_a)-y+(t_a-\tau)\omega , v)d \omega d y  d v d\tau \big|, 
\ee
where    the  kernels $\tilde{\mathcal{K}}^{0}_{k, \tilde{k};l}(y, v, V(t_a))$ and $\tilde{\mathcal{K}}^{1}_{k, \tilde{k};l}(y, v, V(t_a), \omega)$ are defined as follows, 
\begin{multline}\label{2022jan19eqn40}
\tilde{\mathcal{K}}^{0}_{k, \tilde{k};l}(y, v, V(t_a)) : =\int_{\R^3} e^{i y\cdot \xi}  \frac{ V_3(t_a ) (\hat{v}_3 \hat{v}\cdot \xi - \xi_3) }{|\xi|(|\xi| +  \hat{V}(t_a)\cdot \xi)(|\xi|+\hat{v}\cdot \xi)}     \varphi_{l;\bar{l}}(\tilde{v}\times \tilde{\xi})   \varphi_{\tilde{k}}(\slashed \xi) \varphi_k(\xi) d \xi, \\ 
\tilde{\mathcal{K}}^{1}_{k, \tilde{k};l}(y, v, V(t_a), \omega) : =\int_{\R^3} e^{i y\cdot \xi}  \frac{  V_3(t_a )  (|\xi|+\xi\cdot \omega) (\hat{v}_3 \hat{v}\cdot \xi - \xi_3) }{|\xi|(|\xi| +  \hat{V}(s_a)\cdot \xi)(|\xi|+\hat{v}\cdot \xi)}   \varphi_{l;\bar{l}}(\tilde{v}\times \tilde{\xi})   \varphi_{\tilde{k}}(\slashed \xi) \varphi_k(\xi) d \xi. 
\end{multline}

Recall (\ref{sRVM}). We have
 \[
\p_t g(t,x-t\hat{v}, v) = (\p_t +\hat{v}\cdot\nabla_x) f(t,x,v) =-  \nabla_v\cdot \big(E(t,x) f(t,x,v) \big). 
 \]
 Hence, from the above equality,   the Kirchhoff's formula in Lemma \ref{Kirchhoff},  after doing integration by parts in $v$,  we can represent  $  \tilde{\mathcal{H}}^{j;a}_{k, \tilde{k};m,l}(t_1,t_2)$   in terms of kernel as follows, 
\[
 \tilde{\mathcal{H}}^{j ;a}_{k, \tilde{k};m,l}(t_1,t_2)  =  \int_0^{t_a}  \int_{\R^3} \int_{\R^3} \int_{\mathbb{S}^2}  E(\tau,X(t_a)-y+(t_a-\tau)\omega)\cdot \nabla_v\big[\big(\tilde{\mathcal{K}}^{ 0}_{k, \tilde{k};l}(y, v, V(t_a)) + (t_a-\tau) \tilde{\mathcal{K}}^{1}_{k, \tilde{k};l}(y, v, V(t_a)) \big)  \varphi_{j}(v) \big] \] 
 \be\label{jan19eqn35}
 \times   \varphi_{m;-10M_t}(t_a-\tau) f(\tau, X(t_a)-y+(t_a-\tau)\omega, v )    d \omega d y d v  d \tau.
\ee
Recall (\ref{2022jan19eqn40}). After doing integration by parts in $\xi$ in $v$ direction and directions perpendicular to $v$,  we have 
\[
|\tilde{\mathcal{K}}^{ 1}_{k, \tilde{k};l}(y, v, V( t_a), \omega)| + 2^k|\tilde{\mathcal{K}}^{ 0}_{k, \tilde{k};l}(y, v, V(t_a))|
\]
\be\label{june13eqn41}
 \lesssim 2^{2 k+l+ \gamma M_t+150\epsilon M_t}  (1+2^{k-30\epsilon M_t}|y\cdot \tilde{v}|)^{-N_0^3}(1+2^{k+l-30\epsilon M_t}|y\times  \tilde{v}|)^{-N_0^3}, 
\ee 
\[
 \big(|\nabla_v \tilde{\mathcal{K}}^{ 1}_{k, \tilde{k};l}(y, v, V(t_a), \omega)| +  2^k|\nabla_v \tilde{\mathcal{K}}^{ 0}_{k, \tilde{k};l}(y, v, V(t_a))| \big)\varphi_j(v)
\]
\be\label{2022jan19eqn61}
 \lesssim 2^{ 2k +l + \gamma M_t+150\epsilon M_t}  2^{-j-l}    (1+2^{k-30\epsilon M_t}|y\cdot \tilde{v}|)^{-N_0^3}(1+2^{k+l-30\epsilon M_t}|y\times  \tilde{v}|)^{-N_0^3} \varphi_j(v).
\ee

Recall (\ref{jan19eqn36}). Similar to the decomposition we did in (\ref{june14eqn59}) and the obtained estimate (\ref{2022jan19eqn21}), we can also rule out the case when the angle between  $\omega$ and $\pm v$ is relatively big, i.e., the error type. As a result,   from the   estimate  of kernels in (\ref{june13eqn41}) and the estimate (\ref{outsidemajority}) if $|\slashed v|\geq 2^{(\alpha^{\star}+\epsilon)M_t}$,  we have
\[
 \big| \widetilde{End}^{j;a}_{k, \tilde{k};m,l}(t_1,t_2)  \big|\lesssim 1 + 2^{\gamma M_t+  400\epsilon M_t}  \min\{2^{k+l-j}, 2^{-2k-l+\alpha^{\star}M_t+2j} \}  + \big|\int_0^{t_a}  \int_{\R^3} \int_{\R^3} \int_{\mathbb{S}^2}  \varphi_{j}(v)  \p_\tau  \varphi_{m;-10M_t}(t_a-\tau)
\]
\[
 \times \big( \tilde{\mathcal{K}}^{ 0}_{k, \tilde{k};l}(y, v, V(t_a))+ (t_a-\tau) \tilde{\mathcal{K}}^{ 1}_{k, \tilde{k};l}(y, v, V(s_a), \omega )  \big)  {f}(\tau, X(t_a)-y+(t_a-\tau)\omega , v) \phi_{\textup{ess}}( \omega,  v ) d \omega d y  d v d\tau \big|.
\]
From the above estimate and  the Jacobian of changing coordinates $(y_1, y_2, \theta)\longrightarrow X(t_a)-y+(t_a-\tau) \omega$ for the above integral with the kernel $\tilde{\mathcal{K}}^{ 1}_{k, \tilde{k};l}(y, v, V(s_a), \omega )$, where $\omega=(\sin\theta \cos\phi, \sin \theta \sin\phi, \cos\theta)$,  we have
\[
|  \widetilde{End}^{j ;a}_{k, \tilde{k};m,l}(t_1,t_2)  |\lesssim 1+2^{\gamma M_t+  400\epsilon M_t} 2^{\alpha^{\star}M_t/3} + 2^{\gamma M_t+  400\epsilon M_t} \min\{2^{k-j}, 2^{-k-l+2\alpha^{\star}M_t+j}  (2^{m+2l}+2^{-k})\big)\}
\]
\be\label{june14eqn55}
\lesssim   1 + 2^{\gamma M_t+\alpha^{\star}M_t+  400\epsilon M_t}   \lesssim 2^{1.8\gamma M_t}. 
\ee

Now, we focus on the estimate of $\tilde{\mathcal{H}}^{j ;a}_{k, \tilde{k};m,l}(t_1,t_2)$. Recall (\ref{jan19eqn35}). After doing dyadic decomposition for the size of $\slashed v$,   from the estimate of kernels in (\ref{2022jan19eqn61}),  we have
\be\label{2022jan20eqn31}
|\tilde{\mathcal{H}}^{j ;a}_{k, \tilde{k};m,l}(t_1,t_2)|\lesssim  \sum_{\tilde{j}\in [0, j+2]\cap \Z} |\tilde{\mathcal{H}}^{\tilde{j}, j ;a}_{k, \tilde{k};m,l}(t_1,t_2)|, 
\ee
\[
 \tilde{\mathcal{H}}^{\tilde{j}, j ;a}_{k, \tilde{k};m,l}(t_1,t_2):=   \int_0^{t_a} \int_{\R^3} \int_{\R^3}\int_{0}^{2\pi} \int_{0}^{\pi} 2^{  k-j+ \gamma M_t+200\epsilon M_t } \big(2^{m+k} + 1 \big)   \varphi_{m;-10M_t}(t_a-\tau) \| E(\tau,\cdot)\|_{L^\infty_x}  (1+2^{k-30\epsilon M_t}|y\cdot \tilde{v}|)^{-N_0^3} \]
\be\label{june13eqn51}
  \times (1+2^{k+l-40\epsilon M_t}|y\times  \tilde{v}|)^{-N_0}  f(\tau, X(t_a)-y+(t_a-\tau)(\sin\theta\cos\phi,\sin\theta\sin\phi,\cos\theta), v )  \varphi_{\tilde{j}}(\slashed v ) \varphi_{j}(v)  \sin \theta d \theta d \phi  dy d vd \tau. 
\ee

Similar to the obtained estimate  (\ref{june14eqn114}), by using the cylindrical symmetry of solution, the following estimate holds from changing coordinate $\theta\rightarrow X_3(t_a)-y_3 + (t_a-\tau)\cos\theta$,  the space-time estimate (\ref{jan12eqn36}) in Lemma \ref{spacetimeest} and the rough estimate of the electric field (\ref{june2eqn71}) in Lemma \ref{roughgeneral}, 
\[
 \big|  \tilde{\mathcal{H}}^{\tilde{j}, j ;a}_{k, \tilde{k};m,l}(t_1,t_2)\big|\lesssim   \int_0^{t_a} \int_{\R^3} \int_{\R^3}\int_{0}^{2\pi} \int_{0}^{\pi} 2^{  \gamma  M_t+210\epsilon M_t }   2^{-j-l}  \big(2^k(t_a-\tau) + 1 \big)   \| E(\tau,\cdot)\|_{L^\infty_x}  \varphi_{\tilde{j}}(\slashed v ) \varphi_{j}(v) \varphi_{m;-10M_t}(t_a-\tau)
\]
\[
\times \frac{ (1+2^{k+l-40\epsilon M_t}|y_3|)^{-N_0^3}}{|(z_1, z_2)|^{1-2\epsilon}} f(\tau, z_1, z_2, X_3(t_a)-y_3 + (t_a-\tau)\cos\theta)  \sin \theta d \theta d \phi  dz_1 dz_2 dy_3 d vd \tau + 2^{-10M_t}
\]
 \be\label{june14eqn51}
\lesssim    2^{4 {\alpha}^{\star} M_t/3 + M_t/3 +220\epsilon M_t}  2^{-j-2l}    2^{  -2\tilde{j}+j+ \gamma M_t }  + 2^{-10M_t}. 
\ee

Alternatively, from the volume of support of ``$v$'', the estimate (\ref{outsidemajority}) if $|\slashed v|\geq 2^{(\alpha^{\star}+\epsilon)M_t}$,  and the estimate (\ref{june14eqn21}) in Lemma \ref{roughgeneral}, we have
\be\label{june14eqn52}
\textup{(\ref{june13eqn51})} \lesssim 2^{-k- l +\gamma M_t+210\epsilon M_t}  2^{-j-l} \min\{2^{2\tilde{j}+j}, 2^{2\alpha^{\star}M_t + j}\} \int_{0}^{t_a } \| E(\tau, \cdot)\|_{L^\infty_x} d \tau \lesssim  2^{-k- 2l   +(2\gamma +220\epsilon)M_t} \min\{2^{2\tilde{j}}, 2^{2\alpha^{\star}M_t }\}.
\ee

Based on the size of $l$, we split further into two sub-cases as follows. 

$\dagger$\qquad If $l \in [-j + 5M_t/18, -40\epsilon M_t]\cap \Z$.

Note that, for any $(v,\xi)\in supp( \varphi_{\tilde{j}}(\slashed v ) \varphi_{j}(v)   \varphi_{\tilde{k}}(\slashed \xi ) \varphi_{k}(\xi) )$, we have 
\[
\big|\frac{\slashed \xi}{|\xi|}\big| \gtrsim 2^{-30\epsilon M_t}, \quad  |\tilde{v}\times \tilde{\xi}|\sim 2^l, \Longrightarrow 2^{\tilde{j}-j}\sim \big|\frac{\slashed \xi}{|\xi|}\big| \gtrsim 2^{-30\epsilon M_t}-2^{l}\gtrsim  2^{-30\epsilon M_t}.  
\]
 From the above estimate and   the obtained estimate (\ref{june14eqn51}),   we have
\be\label{june14eqn57}
  \big|  \tilde{\mathcal{H}}^{\tilde{j}, j ;a}_{k, \tilde{k};m,l}(t_1,t_2)\big| \lesssim 2^{4  {\alpha}^{\star}  M_t/3 + 4M_t/3 +300\epsilon M_t-2l-2\tilde{j}}\lesssim 2^{2j-2\tilde{j}+4  {\alpha}^{\star}   M_t/3 + 4M_t/3- 5M_t/9 +400\epsilon M_t}     \lesssim 2^{1.8\gamma M_t}.
\ee

$\dagger$\qquad If $l \in [  -40\epsilon M_t,2]\cap \Z$.

After combining the obtained estimates (\ref{june14eqn51}) and  (\ref{june14eqn52}),     we have
\[
  \big|  \tilde{\mathcal{H}}^{\tilde{j}, j ;a}_{k, \tilde{k};m,l}(t_1,t_2)\big| \lesssim \min\{  2^{-k-2l +(2\gamma+220\epsilon)M_t}  \min\{2^{2\tilde{j}}, 2^{2\alpha^{\star}M_t }\},2^{4 {\alpha}^{\star} M_t/3 + 4M_t/3 +200\epsilon M_t-2l- 2\tilde{j} } \}
\]
\be\label{june14eqn58}
\lesssim 2^{-2l-k/2+2{\alpha}^{\star}  M_t/3 + 2M_t/3 +(\gamma+220\epsilon)M_t } \lesssim 2^{1.9\gamma M_t}.  
\ee
Recall (\ref{june14eqn56}) and (\ref{2022jan20eqn31}). After combining the estimates (\ref{june14eqn55}), (\ref{june14eqn57}), and (\ref{june14eqn58}), we have
\be
 \big|H^{j ;2}_{k, \tilde{k};m,l}(t_1,t_2)\big| \lesssim  2^{1.9\gamma M_t+5\epsilon M_t}.  
\ee
Recall the decomposition (\ref{june14eqn63}). Our desired estimate (\ref{june14eqn65}) holds from the above estimate and the estimate (\ref{june14eqn70}). 
 \end{proof}

  Lastly, we give the deferred estimate of $High^{j}_{k, \tilde{k}}(t_1,t_2)$ in the following Lemma. 
 
  \begin{lemma}\label{keylemma3}
Let $j\in [0, (1+\epsilon M_t)]\cap \Z$,  $k\in  [4M_t/5-20\epsilon M_t,3M_t+5\epsilon M_t]\cap \Z$ and $\tilde{k}\in[0, k+2]\cap \Z$, s.t., $\tilde{k}\geq k-30\epsilon M_t$. Under the assumption of Lemma \textup{\ref{bootstraplemma2}}, we have we have
 \be\label{june14eqn72}
 \big| High^{j}_{k, \tilde{k}}(t_1,t_2)\big| \lesssim 2^{(1.9\gamma    +20\epsilon )M_t}. 
\ee
 \end{lemma}
\begin{proof}

Recall (\ref{june10eqn72}) and  (\ref{2022jan29eqn5}). We have 
\begin{multline}\label{june14eqn74}
High^{j  }_{k, \tilde{k}}(t_1,t_2)= \int_{t_1}^{t_2} \langle V(s)\rangle^{-1} E(s, X(s))\cdot  \widetilde{High}^{j  }_{k, \tilde{k}}(s ) d s, \\ 
 \widetilde{High}^{j  }_{k, \tilde{k}}(s ):=    \int_0^s \int_{\R^3} \int_{\R^3} e^{i X(s) \cdot \xi + i(s-\tau)|\xi |}   \mathfrak{M}(\xi, v, V(s))(\hat{v}_3 \hat{v}\cdot \xi - \xi_3)
   \varphi_{\tilde{k}}(\slashed \xi) \varphi_k(\xi) \varphi_{j}(v)\hat{f}(\tau, \xi, v)  d \xi d v  d \tau.
\end{multline}

Recall (\ref{june11eqn33}). Note that, the difference between $ \widetilde{High}^{j  }_{k, \tilde{k}}(s )$ and $H^{j ;2}_{k, \tilde{k}}(t_1,t_2)$ lies only in the symbols and the characteristic time evaluated, which don't play much role in the proof of  the estimate (\ref{june14eqn65}). With minor modifications, we have
\[
\forall s\in [s_1, s_2], \quad | \widetilde{High}^{j}_{k, \tilde{k}}(s )|\lesssim 2^{1.9 \gamma M_t +10\epsilon M_t}. 
\]
Recall (\ref{june14eqn74}). From the above estimate and the estimate (\ref{june14eqn21}) in Lemma \ref{roughgeneral}, we have 
\be
 \big| High^{j  }_{k, \tilde{k}}(s_1,s_2)\big| \lesssim \int_{s_1}^{s_2} 2^{-\gamma M_t  }\| E(s, \cdot)\|_{L^\infty} 2^{1.9 \gamma M_t +10\epsilon M_t} d s \lesssim 2^{(1.9\gamma    +20\epsilon )M_t}. 
\ee
Hence finishing the proof of our desired estimate (\ref{june14eqn72}).

 \end{proof}

 \end{document}